\documentclass[11pt, oneside]{article}   	
\usepackage{geometry}                		
\usepackage{graphicx}				
\usepackage{xcolor}					 
\usepackage{comment}
\usepackage{tikz}
\usepackage{hyperref}
\usepackage{cite}
		
\usepackage{amssymb}
\usepackage{amsmath}
\usepackage{amsthm}

\newtheorem{theorem}{Theorem}[section]
\newtheorem{lemma}[theorem]{Lemma}
\newtheorem{prop}[theorem]{Proposition}
\newtheorem{cor}[theorem]{Corollary}

\theoremstyle{definition}

\newtheorem{example}[theorem]{Example}

\title{Exploring the infinitesimal rigidity of planar configurations of points and rods}
\author{Signe Lundqvist, Klara Stokes and Lars-Daniel \"Ohman\\Ume{\aa} University, Sweden}

\begin{document}
\maketitle

\begin{center}
\textbf{Abstract}
\end{center}

This article is concerned with the rigidity properties of geometric realizations of incidence geometries of rank two as points and lines in the Euclidean plane; we care about the distance being preserved among collinear points. 
We discuss the rigidity properties of geometric realizations of incidence geometries in relation to the rigidity of geometric realizations of other well-known structures, such as graphs and hypergraphs.

The $2$-plane matroid is also discussed.

Further, we extend a result of Whiteley to determine necessary conditions for an incidence geometry of points and lines with exactly three points on each line, or 3-uniform hypergraphs, to have a minimally rigid realization as points and lines in the plane. We also give examples to show that these conditions are not sufficient.

Finally, we examine the rigidity properties of $v_k$-configurations. We provide several examples of rigid $v_3$-configurations, and families of flexible geometric $v_3$-configurations. The exposition of the material is supported by many figures.

\section{Introduction}

This article is concerned with realizations of incidence geometries that consist of points and lines in the Euclidean plane. We will view the lines as rigid bodies, so that the pairwise distance between points on a line is preserved. If there are two points on each line, the problem reduces to the well studied rigidity theory of graphs in the plane. We think of such realizations as consisting of rods (the rigid lines).

A line with $k$ points determines $\binom{k}{2}$ distance constraints, as the distance between any pair of points on the line needs to be preserved. One possible model for a rigid body with $k$ points is a geometric realization of a rigid graph on at least $k$ vertices. Replacing each body with such a rigid graph gives a description of a body and joint framework  in terms of a geometric realization of a graph. If we consider a realization where the $k$ vertex points on each body are collinear, we get a graph model of a rod configuration. However, in this model  the theory of infinitesimal rigidity does not generalize well. Consider for example a body on three vertices; it will be represented by a triangle graph that is geometrically realized on a line. We would like a model in which a body with three vertices is infinitesimally rigid, but the triangle graph in this special position has a non-trivial infinitesimal motion.
This infinitesimal motion does not come from a continuous motion of the triangle graph, but is a mere artifact of the model.

Whiteley made early important contributions to the understanding and the combinatorial characterization of the infinitesimal rigidity of rod configurations \cite{Whiteley89,Whiteley1996}.
He classified the minimally infinitesimally rigid rod configurations, using a notion of minimality that says that a configuration in which three points are covered in pairs by two rods may be smaller than if the two rods are replaced with one rod covering all three points.

In this article we will instead say that a rod configuration is minimally rigid if no line can be removed without the rod configuration becoming flexible. This is a natural generalization of the notion of minimal rigidity for graphs; a (geometric realization of a) graph is minimally rigid if no edge can be removed without the result being flexible.

Further related research has given combinatorial characterizations of the rigidity of body and hinge frameworks with all joints incident to a body lying in a common hyperplane, in all dimensions; these results are known as the {\em molecular conjecture}.
These results are however only valid when the number of bodies meeting at each point is two.

In the classical litterature, the study of configurations of points and lines has been mostly concerned with those configurations that have the same number $k$ of points on every line, and the same number $r$ of lines through every point. Such configurations are called $(v_r,b_k)$-configurations and when $r=k$ they are called $v_k$-configurations.
If either of $r$ or $k$ equals two, then either the configuration or its dual is essentially a graph. Therefore the literature on configurations is mostly concerned with the case when $r,k\geq 3$, a case which the available results on minimal rigidity of rod configurations do not cover. 

This article was written with the aim of shedding some light on this problem, motivating the reader to explore the topic of configurations within the scope of rigidity theory. Configurations of points and lines are important in geometry. Historically, they have fascinated many now famous mathematicians and they keep fascinating people also today.

\section{Background}
\label{Background}

\subsection{Graphs, configurations and incidence geometries}
\label{sec:structures}

An incidence structure is a quadruple $(V,T,t,I)$, where $V$ is a set of varieties, $T$ is a set of types, $t:V\rightarrow T$ is an exhaustive function assigning a type to each element of $V$, and $I$ is an incidence relation on $V$, such that no two elements of the same type are related. The rank of the incidence geometry is its number of types. 
The incidence relation on $V$ defines a multipartite graph with vertex set $V$ and parts corresponding to the types, called the incidence graph of the incidence structure. 

A flag of an incidence structure is a clique of the incidence graph. By
definition, all vertices in a flag have distinct types. An incidence geometry is an incidence structure with the property that every maximal flag contains an element of each type.
If the incidence structure has rank two, then clearly the incidence structure is a geometry if and only if its incidence graph has no isolated vertices. In particular, all connected incidence structures of rank two are incidence geometries.

The real Euclidean space $\mathbb{E}^n(\mathbb{R})$  defines an incidence geometry $\mathcal{E}^n$ of rank $n$, by taking as varieties the linear varieties of $\mathbb{E}^n(\mathbb{R})$ (points, lines, planes, \dots), and defining two varieties to be incident if one is included in the other.
A linear realization $\rho$ of dimension $m$ of an incidence geometry $\Gamma$ of rank $n$ is a function $\rho:\Gamma\rightarrow \mathcal{E}^m$ that preserves incidence. In this article, we will only be concerned with linear realizations of dimension two.

A Euclidean geometric configuration of points and lines is an incidence geometry of rank two, together with a linear realization in terms of points and lines in  Euclidean space. 
If the linear realization is injective, then the incidence geometry must have the property that every pair of elements of one type is simultaneously incident with at most one element of the other type. This property has also been called linearity of the incidence geometry \cite{Izquierdo-Stokes}, since it captures the abstract notion of the incidences of a line arrangement. 

Historically, the literature has mostly been concerned with geometric configurations of points and lines in which all lines are incident with the same number $k$ of points and all points are incident with the same number $r$ of lines. Such configurations are called $(v_r,b_k)$-configurations, where $v$ is the number of points and $b$ is the number of lines. If $r=k$, then also $v=b$, and the configuration is called balanced, in which case the name $v_k$-configuration is used.  Many important examples of configurations of points and lines are balanced, such as the $v_3$-configurations of Pappus and Desargues. 
Incidence geometries with the same abstract properties as those giving injective linear realizations that are geometric $(v_r,b_k)$-configurations or $v_k$-configurations are in the literature also called combinatorial configurations \cite{grunbaum, pisanski}.

From this point on, we will restrict our attention only to rank two incidence geometries.  We will use the notation $(P,L,I)$ to denote an incidence geometry of rank two, where $P$ and $L$ are the two sets of varieties of distinct type, and $I$ is the incidence relation. 

There is a direct correspondence between incidence geometries of rank two and hypergraphs. If all elements of one type (the `points') are incident to exactly two elements of the other type (the `lines'), then the incidence geometry is an ordinary graph. The combinatorial $(v_r,b_2)$-configurations are the $r$-regular graphs. The above definitions do not
exclude the possibility that two lines are incident with the exact same
set of points, so the corresponding graphs might be multigraphs (with
repeated edges, that is), but we will only be considering the case where
the graphs are simple.

Because any two points define a line in Euclidean space, any assignment of points to the vertices of a graph give rise to a linear realization of the graph. If the incidence geometry is not a graph, an injective linear realization can be difficult to find, and may not even exist. For example, there is no injective linear realization of the unique $7_3$-configuration, also known as the Fano plane. An incidence geometry always has trivial linear realizations, in which all points are mapped to the same point and all points are mapped to the same line. In the trivial linear realizations, there will however be additional incidences. 

The topic of this article is motions of linear realizations of rank two incidence geometries. We will also consider other geometric realizations of rank two incidence geometries: body and joint realizations and string configurations.

A body and joint realization, or a body and joint framework, of an incidence geometry of rank two $S=(P,L,I)$ is an assignment of points in Euclidean space to the elements of $P$. The elements of $L$ are then thought of as any rigid body that contains the points that are incident to it, and the elements of $P$ are thought of as joints, around which the bodies can rotate freely.

The incidence geometry may be realized in other ways as well. One such way is as follows: First define a graph on the vertex set $P$ by adding edges forming a tree for each element in $L$, and then consider a linear realization of this graph with the property that the edges in a tree corresponding to an element in $L$ are all collinear. Such a realization is called a string configuration.
It is important to note that a string configuration is a geometric realization of a graph with certain edges collinear, rather than a geometric realization of the incidence geometry in terms of points and lines.

A string configuration of the incidence geometry cannot exist without there being a corresponding linear realization of the same incidence geometry. Given a linear realization of the incidence geometry, the edges $(p,q)$ in the graph that constitutes the string configuration are placed along the line spanned by $p$ and $q$ in the linear realization.
There may however be several string configurations coming from the same linear realization of an incidence geometry; by choosing distinct  tree graphs representing the  lines, different string configurations are obtained.

\subsection{The $2$-plane matroid, concurrence geometries and parallel redrawings}

Whiteley introduced the $k$-plane matroid as a generalisation of the ``picture matroid'', as a tool in scene analysis \cite{Whiteley89, SerWhi}.
Another application of the $k$-plane matroid is in the study of infinitesimal rigidity of rod configurations.

Consider an incidence geometry $S=(P,L,I)$. We want to study the set of all linear realizations of the incidence geometry in the Euclidean plane with specified line slopes.
Some linear realizations will be degenerate in that some points coincide. Crapo showed that the linear realizations of $S$ with different degeneracies form a combinatorial lattice \cite{Crapo84}.

The collinearity of three finite points $p$, $q$ and $r$ with projective coordinates $(x_p:y_p:1)$, $(x_q:y_q:1)$ and $(x_r:y_r:1)$ defines an equation $$\det\left(\begin{array}{ccc}x_p&x_q&x_r\\y_p&y_q&y_r\\1&1&1\end{array}\right)=0.$$
If a line in the configuration is incident with exactly three points, then the line defines one such equation.
More generally, if a line $\ell$ in the configuration is incident with $m=m(\ell)$ points $p_1,\dots,p_m$, then the triples $\{p_1,p_2,p_i\}$ for $i\in \{3,\dots m\}$ define $m-2$ such equations.
(A line that is incident with fewer than three points defines no such equation.)

For every incidence $(p,\ell)\in I$, where $\ell=\langle q,r\rangle$ for any points $q$ and $r$ on $\ell$ such that $p\not\in \{q,r\}$, create a new equation from the equation
$$
\det\left(\begin{array}{ccc}x_p&x_q&x_r\\y_p&y_q&y_r\\1&1&1\end{array}\right)=x_p(y_q-y_r)+y_p(x_r-x_q)+x_qy_r-x_ry_q=0$$
  by first rewriting it as
  $$x_p\frac{y_q-y_r}{x_q-x_r}-y_p+\frac{x_qy_r-x_ry_q}{x_r-x_q}=0$$
and then fixing the  slope $\delta=\frac{y_q-y_r}{x_q-x_r}$. 

For a fixed line slope, the new equation has three indeterminates: the planar coordinates $x_p$, $y_p$ of the point $p$ and the $y$-intercept $\frac{x_qy_r-x_ry_q}{x_r-x_q}$ of the line $\ell$. 

A total of $|I|$ equations is obtained in this way, defining a system of equations in $|L|+2|P|$ indeterminates: one $y$-intercept for every line (lines are assumed to not be vertical) and two coordinates for every point.

Suppose $\rho$ is a linear realization of $S$. Then $\rho$ gives a set of line slopes, that we may fix. Conversely, any set of line slopes determines linear realizations, that may be trivial, meaning that all points are given the same coordinates. 

For an incidence geometry $S$ with a realization $\rho$, we call the $|I|\times \left(|L|+2|P|\right)$ matrix $M(S,\rho)$ defining the system of equations the concurrence geometry matrix of the incidence geometry with the given realization $\rho$.

A parallel redrawing of an incidence geometry with a fixed slope $f_j$ for every line $\ell_j$, is an assignment of a point $(x_i, y_i)$ to each $p_i \in P$ and an assignment of a number $h_j$ to each $\ell_j \in L$ such that if $( p_i,\ell_j) \in I$, 

\begin{equation}
f_jx_i + y_i + h_j = 0.
\label{parallel redrawing}
\end{equation}

As defined, a parallel redrawing of an incidence geometry with fixed line slopes is simply a linear realization of the incidence geometry in the plane, with the given line slopes. 

Furthermore, given a realization $\rho$ of $S$, the kernel of the matrix $M(S,\rho)$ consists of the set of triples $(x_i,y_i,h_j)$ that satisfy Equation \ref{parallel redrawing} for the set of slopes given by $\rho$. Hence, given a linear realization $\rho$ of the incidence geometry, the matrix $M(S,\rho)$ gives a set of linear realizations with the line slopes defined by $\rho$.

For any set of line slopes there is a space of trivial linear realizations in which all points have the same coordinates. This space has dimension two, corresponding to the two coordinates, which then also determine the $y$-intercept. 

If we can realize an incidence geometry with a given set of line slopes so that at least two points have different coordinates, there is a three-dimensional space of parallel redrawings, generated by two translations and one dilation \cite{Whiteley89}. It follows that if we can realize the incidence geometry in such a way that at least two points have distinct coordinates, the kernel of $M(S,\rho)$ has dimension at least three.

We say that a linear realization is proper if all combinatorial points are realized with distinct pairs of coordinates. In particular, the kernel of $M(S,\rho)$ will have dimension at least three for a proper linear realization with more than two points. 

So, if $S$ has a proper linear realization such that the rows of the concurrence geometry matrix $M(S,\rho)$ are independent, then necessarily $|I| \leq |L| + 2|P| -3$, since the kernel of $M(S, \rho)$ always has dimension at least three.

The $2$-plane matroid is a matroid defined on the set of incidences $I$ of an incidence geometry $(P,L,I)$ in terms of independent sets as follows: $I$ is independent if $|I'|\leq |L'|+2|P'|-2$, for any nonempty subset $I' \subset I$, where $P'\times L'\subseteq P\times L$ is the support of $I'$.

In some contexts, for example if we are considering graphs in the plane, rigidity can be described by the matroid defined in terms of the rows of a rigidity matrix. See \cite{GraSerSer} for more background on matroids and their use in combinatorial rigidity theory. Similarly, the $2$-plane matroid is related to the row matroid of the matrix $M(S,\rho)$.

If a set of incidences is independent in the $2$-plane matroid, then the rows of $M(S, \rho)$ are independent for almost all realizations $\rho$. This is essentially Theorem 4.1 in \cite{Whiteley89}. This means that if $|I|=|L|+2|P|-2$, $S$ will not have a proper linear realizations for most choices of line slopes $\rho$. More specifically, any choice of realization $\rho$ such that the rows of $M(S, \rho)$ are independent, will yield a linear realization with a two-dimensional space of parallel redrawings. This means that $\rho$ must be trivial, and so must any realization with the same line slopes.

\subsection{Notions of rigidity}
In this subsection we survey the distinct notions of rigidity that corresponds to the different geometric realizations of incidence geometries in Section \ref{sec:structures}. 

\subsubsection{Notions of rigidity for rod configurations and graphs}

We say that a linear realization of a rank two incidence geometry as points and lines in the Euclidean plane is (continuously) flexible if there is a continuous motion of some of its points and lines, other than the Euclidean motions of the entire configuration,  that preserves the incidences of the configuration and the distances between the points on the same line. We say that a configuration of points and lines is (continuously) rigid in the Euclidean plane if it is not flexible.

In the motions we consider, points on the same line never move in relation to each other; the lines are rigid bodies. Therefore it is natural to think of such linear realizations as configurations of rods (the lines) and pin-joints (the points), and call them rod configurations. A rod configuration is therefore a linear realization as points and lines, together with the motions of the linear realization. Rod configurations
can be seen as special cases of body and joint frameworks.

We say that a rod configuration is minimally rigid if no rod can be removed from the configuration without it becoming flexible. Note that we do not allow the removal of joints, and if a joint belongs to only one rod, then the removal of that rod would result in a flexible configuration, since that joint would then be able to move independently of the rest of the rod configuration.

An infinitesimal motion of a rod configuration is an assignment of a vector $v \in \mathbb{R}^2$ to each point $p \in P$ such that restricted to each rod, the vectors define the linear part of a Euclidean rigid motion.
A rod configuration is infinitesimally rigid if any infinitesimal motion of the rod configuration is the linear part of a Euclidean rigid motion. We call such infinitesimal motions the trivial infinitesimal motions of the rod configuration.

For planar rod configurations with at least two distinct points there are three independent trivial infinitesimal motions, coming from the three generators of the Euclidean planar group: one rotation and two translations. If there is a non-trivial infinitesimal motion, then the rod configuration is said to be infinitesimally flexible.

\subsubsection{Characterizing rigidity of graphs in the plane}

In the special case where the incidence geometry is a graph $G=(V,E)$, then a planar linear realization (the rod configuration) of $G$ is a so-called framework of the graph. In this case, the infinitesimal motions of the framework $(G,\rho)$ is an assignment of a vector $m_i \in \mathbb{R}^2$ to the point $\rho(v_i)$ for $v_i\in V$ such that  
\[
	{(m_i-m_j)^T}\cdot(\rho(v_i)-\rho(v_j))=0
\]
for all edges $(v_i,v_j)\in E$.

The following lemma relates infinitesimal rigidity to continuous rigidity.

\begin{lemma}[\cite{Gluck}]
If a framework $\rho$ of a graph $G$ is infinitesimally rigid, then it is rigid. 
\label{inf rig}
\end{lemma}

It is well known that the algebraic dependencies among the points assigned to the vertices can affect the flexibility and the infinitesimal flexibility of a graph realized in the plane.
A framework of a graph is called generic if its set of point coordinates is algebraically independent. The converse of Lemma \ref{inf rig} is not true in general, but it holds for generic frameworks \cite{Roth}. So for generic frameworks, rigidity is equivalent to infinitesimal rigidity.

Furthermore, by the following lemma, it makes sense to talk about generic rigidity of a graph. 

\begin{lemma} [\cite{Gluck, Graphs and Geometry}]
Let $G=(V,E)$ be a graph. If there is some infinitesimally rigid framework of a graph, then any generic framework of $G$ is rigid.
\label{generic rigidity}
\end{lemma}

We say that a graph is generically rigid if all its generic frameworks are infinitesimally rigid, or, equivalently, rigid.
A graph is generically minimally rigid, if it is generically rigid, and the removal of any edge results in a graph that is not generically rigid. A famous result due to Geiringer, and later to Laman, says that in the Euclidean plane, all generically minimally rigid graphs with a given number of vertices have the same number of edges.

\begin{theorem}[Geiringer-Laman, \cite{PolGei1927, Laman}]
	Let $G=(V,E)$ be a graph. Then $G$ is generically minimally rigid in the Euclidean plane if and only if 
	\begin{itemize}
	\item $|E|=2|V|-3$
	\item $|E'| \leq 2|V'|-3$ for any nonempty subset $E' \subseteq E$, where $V'$ is the set of vertices in the subgraph generated by $E'$.
	\end{itemize}
  \label{thm_geiringer}
\end{theorem}
 
For linear realizations of graphs in the plane, there is a one-to-one correspondence between infinitesimal motions and parallel redrawings \cite{Crapo85, Whiteley88, CraWhi}. The trivial infinitesimal motions correspond to the trivial parallel redrawings.

\subsubsection{Characterizing rigidity of string configurations in the plane}

A string configuration realizing an incidence geometry $S=(P,L,I)$ is infinitesimally rigid if when considered as a framework of a graph, it is infinitesimally rigid.
Whiteley proved the following result, characterizing which incidence geometries have realizations as minimally infinitesimally rigid string configurations.

\begin{theorem}[Whiteley \cite{Whiteley89}]
An incidence geometry $S=(P,L,I)$ has a realization as a minimally infinitesimally rigid string configuration if and only if 
\[
	|I| = |L| + 2|P|-3 
\]
and 
\[
	|I'| \leq |L'| + 2|P'| - 3
\]
for any proper subset $I' \subset I$.

\label{Whiteley_5.2}
\end{theorem}

If the incidence geometry is a graph $G=(V,E)$, so that $P=V$ and $L=E$, then $|I|=2|E|$. In this case, Theorem \ref{Whiteley_5.2} is simply the Geiringer-Leman Theorem, Theorem \ref{thm_geiringer}.

Whiteley proved Theorem \ref{Whiteley_5.2} using parallel redrawings and the concurrence geometry matrix. As a string configuration is a framework of graph, the parallel redrawings of a (planar) linear realization of $S$ are in one-to-one correspondence with the infinitesimal motions of its realizations as a string configuration with the same line slopes. 

We say that a string configuration is independent if the rows of its concurrence geometry matrix are independent. An incidence geometry has a realization as an independent string configuration if and only if $|I'| \leq |L'| + 2|P'| -3$ for all subsets of incidences $I'$ \cite{Whiteley89}.

A key point in the proof of Theorem \ref{Whiteley_5.2} is that if the incidence geometry satisfies the conditions in Theorem \ref{Whiteley_5.2}, then it has a proper linear realization for almost all choices of normals. As previously mentioned, it is not always true that an incidence geometry has proper linear realizations with generic normals. In fact, this is true if and only if the count of Theorem \ref{Whiteley_5.2} holds \cite{Whiteley89}. 

\subsubsection{Rigidity of body and joint frameworks and rod configurations in the plane}

A body and joint framework can be represented in terms of frameworks of graphs by replacing each body by a minimally infinitesimally rigid framework of a graph with vertex set including the points/joints of the body. This gives a framework of a graph representing the body and joint framework.

A body and joint framework is independent if the framework modeling the body and joint framework is independent as a string configuration, and (minimally) infinitesimally rigid if it is an (minimally) infinitesimally rigid string configuration. Recall that a string configuration is a framework of a graph in a (possibly) non-generic position that is determined by the incidence geometry, and that the string configuration is (minimally) infinitesimally rigid if it is so as a framework of a graph.

Whiteley gave a combinatorial characterization of minimal infinitesimal rigidity of body and joint realizations of incidence geometries, thereby generalizing Theorem \ref{thm_geiringer} to hypergraphs.
He also showed that an incidence geometry that has a minimally infinitesimally rigid body and joint realization, also has a realization as a minimally infinitesimally rigid rod configuration.

\begin{theorem}[Whiteley, \cite{Whiteley89}]
Given an incidence geometry $S=(P,L,I)$ the following are equivalent:
	\begin{enumerate}
	\item $S$ has an independent (minimally infinitesimally rigid) body and joint realization in the Euclidean plane.
	\item $S$ satisfies $2|I| \leq (=) 3|L| + 2|P| -3$, and for every subset of bodies with the induced subgraph of attached joints $2|I'| \leq 3|L'|+2|P'|-3$.
	\item $S$ has an independent (minimally infinitesimally rigid) body and joint realization in the Euclidean plane such that each body has all its joints collinear.
	\end{enumerate}
	\label{Whiteley5.4}
\end{theorem}

In this article, unless otherwise stated, we say that a rod configuration is minimally rigid if removing any line results in a flexible rod configuration. Our notion of minimal infinitesimal rigidity of rod configurations is not the same as the notion that apppears in statement 3 of Theorem 2.5. In short, there are incidence geometries that have realizations as minimally infinitesimally rigid rod configurations in our context, but for which there is no minimally infinitesimally rigid body and joint realization such that each body has all its joint collinear. However, Theorem \ref{Whiteley5.4} shows that the converse is true; any incidence geometry that can be realized as a minimally infinitesimally rigid body and joint framework can also be realized as a minimally infinitesimally rigid rod configuration.

Furthermore, as seen in \cite{Projective Lens}, if $S$ has a realization as a rod configuration, it is possible to construct a body and joint framework (for example by replacing each rod with a cone on the points incident to the rod) with the same rigidity properties, infinitesimal and continuous. In fact, this is the same body and joint framework that is constructed in \cite{Whiteley89} to prove the implication $2 \implies 3$ of Theorem \ref{Whiteley5.4}. Hence, if $S$ has a realization as an infinitesimally rigid rod configuration, then $S$ has a realization as an infinitesimally rigid body and joint framework.

 As an example, the incidence geometry realized as a rod configuration in Figure \ref{2P-2 Example} is minimally rigid as a rod configuration, however it does not satisfy condition 2 of Theorem \ref{Whiteley5.4}, and does not have a realization as a minimally rigid body and joint framework. It does however have a realization as an infinitesimally rigid body and joint framework, which is not minimally infinitesimally rigid.

Tay and Whiteley independently characterized which incidence geometries have realizations as rigid body and hinge frameworks in $\mathbb{R}^d$, where a body and hinge framework is a body and joint framework such that each joint is incident to at most two bodies \cite{Whiteley88,Tay1984}. Tay and Whiteley jointly conjectured in \cite{TayWhi1984} that any incidence geometry that can be realized as a rigid body and hinge framework in $\mathbb{R}^d$ can be realized as a rigid body and hinge framework in $\mathbb{R}^d$ such that all hinges incident to a body lie in a common hyperplane. This is known as the molecular conjecture.

A special case of the molecular conjecture in the plane follows from Theorem \ref{Whiteley5.4}; namely the special case where the rigidity of the body and hinge framework is minimal. However Theorem \ref{Whiteley5.4} holds for general body and joint frameworks, not only body and hinge frameworks, allowing more than two bodies to meet at a point. 

Jackson and Jord\'an proved in \cite{JacJor08} that the molecular conjecture holds in the plane, and Katoh and Tanigawa proved in \cite{KatTan11} that the molecular conjecture holds in general.
Jackson and Jord\'an therefore solved the question about minimal rigidity for planar rod configurations in the special case when each point is incident to two lines only. 
Body and hinge structures are further studied in \cite{JacJor10} and \cite{Tay1989}.
None of these results solve the question of minimal infinitesimal rigidity for rod configurations in its generality. 

Frameworks of graphs that remain rigid with a given set of three points collinear have been classified by Jackson and Jord\'an in \cite{JacJor05}. Their result was extended to sets of points of arbitrary size by Eftekhari et. al. in \cite{Frameworks2019}.

\section{Minimal rigidity of rod configurations in the plane}
\label{Minimal rigidity}

The minimally infinitesimally rigid graphs in the plane are the bases (maximally independent sets) of the rigidity matroid defined on the set of edges of a complete graph. The independent sets of the rigidity matroid correspond to linearly independent rows of the rigidity matrix. See \cite{GraSerSer} for some background on matroids and their use in combinatorial rigidity theory.

The independent sets of the $2$-plane matroid  also correspond to linearly independent rows of a matrix, namely the rows of the matrix $M(S, \rho)$. However, if $S=(P,L,I)$ is an incidence geometry of rank two such that $I$ is maximally independent in the $2$-plane matroid, then $S$ has only the trivial linear realizations for almost all choices of line slopes. Whiteley characterized the incidence geometries that have minimally infinitesimally rigid realizations as string configurations, see Theorem \ref{Whiteley_5.2}. These incidence geometries are certainly independent in the $2$-plane matroid, but not maximally independent.

In this section we will give examples of planar rod configurations that are infinitesimally minimally rigid in another (rather natural) way, but again they do not correspond to bases of the $2$-plane matroid. 

We say that a rod configuration is minimally (continuously/infinitesimally/globally) rigid if it is rigid, but the removal of any rod results in a flexible rod configuration. 
When we remove a rod, we remove only the rod and no points. Note that a point which is not on any line can move independently of the rest of the rod configuration. Therefore, if we remove a line such that there is some point of the configuration incident only to that line, then the resulting rod configuration is flexible.

We wish to understand the rigidity properties of geometric $v_k$-configurations. As a step in this direction, we focus on incidence geometries in which all lines are incident to the same number of points.

We say that an incidence geometry $S=(P,L,I)$ is $k$-uniform if every line in $L$ is incident to exactly $k$ points. In a $k$-uniform incidence geometry $|I|= k|L|$, in which case the implication $2 \implies 3$ of Theorem \ref{Whiteley5.4} can be restated as follows:

\begin{cor}
Let $S=(P,L,I)$ be a $k$-uniform incidence geometry. If 
\begin{itemize}
\item $(2k-3)|L|=2|P|-3$ and
\item $(2k-3)|L'| \leq 2|P'|-3$ for every subset $L' \subseteq L$, 
\end{itemize}
then $S$ has a realization as a  minimally infinitesimally rigid rod configuration.
\label{Point-line count}
\end{cor}

In the setting of Theorem \ref{Whiteley5.4} the other implication also holds. However, in our context there are examples of rod configurations that are minimally infinitesimally rigid, but that do not satisfy the count of Theorem \ref{Whiteley5.4} and Corollary \ref{Point-line count}, see for example Figure \ref{2P-2 Example}. 

Note that $v_k$-configurations always have $|I|=k|L|$ and $|L|=|P|$, so $(2k-3)|L|=2|P|-3$ is never true for $v_k$-configurations with $k \geq 3$. If $k=2$, then $|P|=|L|=3$ is the only solution. However, $v_k$-configurations may have subconfigurations that satisfy the count given in Corollary \ref{Point-line count}.

\begin{figure}
	\begin{center}
	\begin{tikzpicture}
\filldraw[black] (0,0) circle (1pt);
\filldraw[black] (2,0) circle (1pt);
\filldraw[black] (4,0) circle (1pt);

\filldraw[black] (0,-2) circle (1pt);
\filldraw[black] (2,-2.5) circle (1pt);
\filldraw[black] (4,-3) circle (1pt);
\filldraw[black] (0.88,-1.11) circle (1pt);
\filldraw[black] (1.6,-1.2) circle (1pt);
\filldraw[black] (2.91,-1.36) circle (1pt);

\draw[thin] (0,0) -- (4,0);
\draw[thin] (0,-2) -- (4,-3);
\draw[thin] (0,0)--(2,-2.5);
\draw[thin] (2,0) -- (4,-3);
\draw[thin] (4,0) -- (0,-2);

\draw[thin, dashed] (0,0) -- (4,-3);
\draw[thin, dashed] (2,0) -- (0,-2);
\draw[thin, dashed] (4,0) -- (2,-2.5);
\draw[thin, dashed] (0.88, -1.11) -- (2.91,-1.36);

\end{tikzpicture}
	\end{center}
	\caption{A minimally rigid subconfiguration of the Pappus configuration.}
	\label{Pappus subconfiguration}
\end{figure}
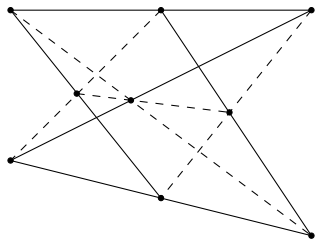

\begin{example}[Pappus configuration]
  Consider the rod configuration in Figure \ref{Pappus subconfiguration}. The union of the lines and the dotted lines form the Pappus configuration. By removing the dotted lines we obtain a subconfiguration  that clearly is minimally continuously rigid. Indeed, we cannot remove any of the lines with a point incident only to that line without the rod configuration becoming flexible. Removing either of the other two lines also  results in a flexible rod configuration. The existence of this minimally continuously rigid configuration tells us that Pappus configuration must be continously rigid in the position shown in Figure \ref{Pappus subconfiguration}.  

The same subconfiguration  satisfies the count in Corollary \ref{Point-line count}. It follows that it has at least one realization as a minimally infinitesimally rigid rod configuration. 

As we can see in Figure \ref{Pappus subconfiguration}, at least one of those realizations extends to a geometric realization of the Pappus configuration. We can therefore conclude that Pappus configuration has at least one realization as an infinitesimally rigid rod configuration.
\end{example}

It is not necessarily true that a rod configuration realizing a subconfiguration extends to a rod configuration realizing the whole configuration. For example the Fano plane has a subgeometry that can be realized as a minimally infinitesimally rigid rod configuration (see Example \ref{2P-2 ex}). However, any rod configuration of this subconfiguration that extends to a rod configuration of the Fano plane  must have all points in the same position, or all rods along the same line.

By the definition of minimal rigidity of rod configurations, it is clear that any rigid rod configuration has at least one minimally rigid subconfiguration. A minimally rigid subconfiguration of a $k$-uniform incidence geometry can only satisfy the counting condition in Corollary \ref{Point-line count} in certain cases, namely when $2|P|-3$ is divisible by $2k-3$.

If $k=2$, so that the incidence geometry is a graph, then $2|P|-3$ is always divisible by $2k-3=1$, and Corollary \ref{Point-line count} is the classical Geiringer-Laman Theorem, Theorem \ref{thm_geiringer}.

If $k=3$, then $2k-3=3$, and for any given $P$, one of $2|P|-1$, $2|P|-2$ and $2|P|-3$ is divisible by $3$. This gives a lower bound on how many lines a 3-uniform incidence geometry needs to have a realization as an infinitesimally rigid rod configuration, as the next proposition shows.

\begin{prop}
Let $S=(P,L,I)$ be a $3$-uniform incidence geometry such that $S$ has a realization as an infinitesimally rigid rod configuration. Assume that $3|L|=2|P|-3$, $3|L|=2|P|-2$ or $3|L|=2|P|-1$. Then $S$ has a realization as a minimally infinitesimally rigid rod configuration.
\label{minrig}
\end{prop}

\begin{proof}
  Assume for a contradiction that it is possible to remove some line $\ell$ from $S$ to obtain an infinitesimally rigid rod configuration. Let $L'=L \setminus \{ \ell \}$ and consider $S=(P', L', I')$. If $\ell$ can be removed without the rod configuration becoming infinitesimally flexible, it must hold that $|P'|=|P|$. Therefore $3|L'|= 3(|L|-1)=3|L|-3 \leq 2|P|-4=2|P'|-4$, where the inequality holds as $3|L| \leq 2|P|-1$. Then by Theorem \ref{Whiteley5.4}, $S'$ does not have a realization as an infinitesimally rigid body and joint framework and therefore $S'$ cannot have a realization as an infinitesimally rigid rod configuration.
Hence the infinitesimally rigid rod configuration realizing $S$ is minimally infinitesimally rigid.
\end{proof}

\begin{figure}
  \begin{center}
    \begin{tabular}{cc}
      \includegraphics[scale=0.2]{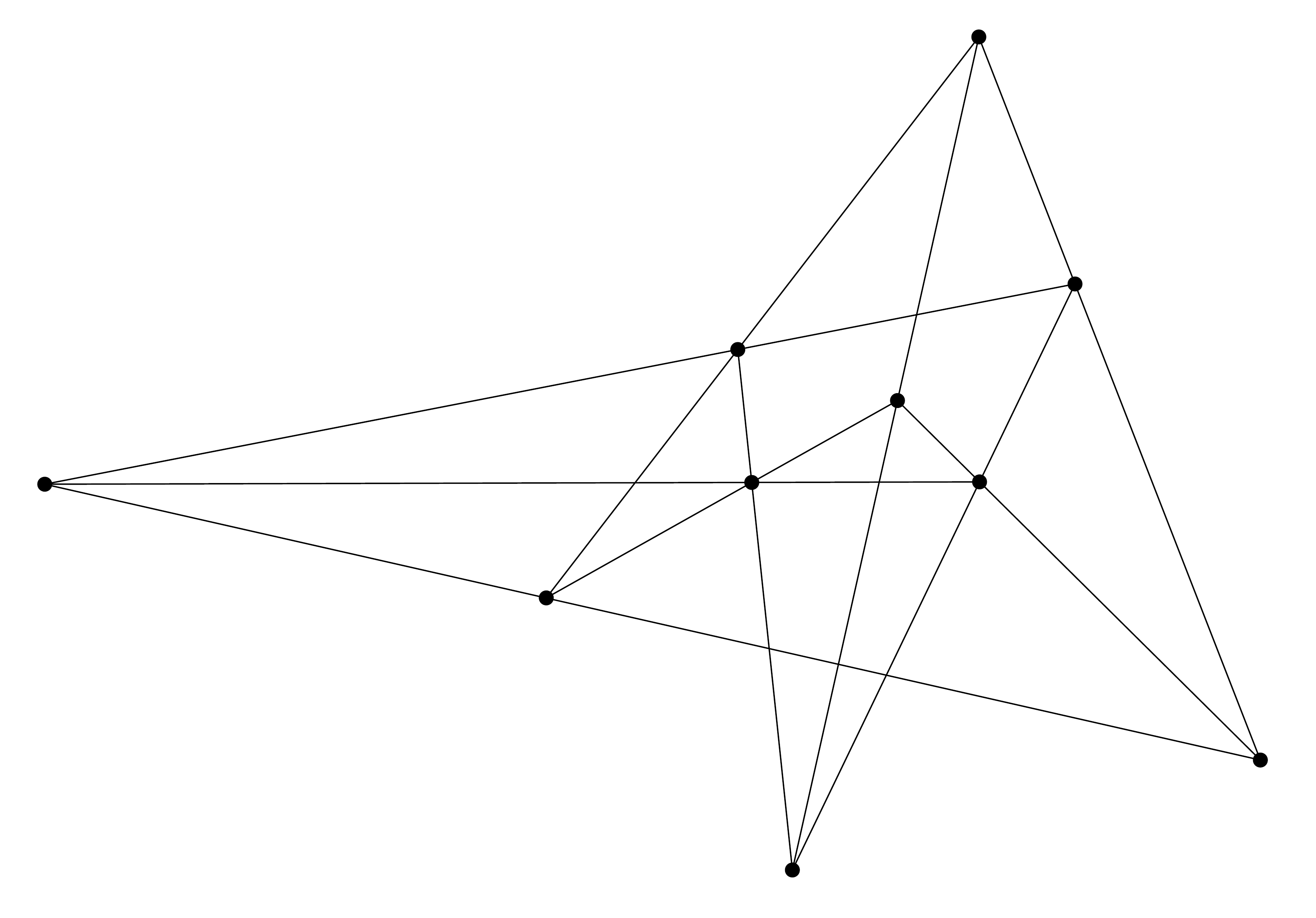}&  \includegraphics[scale=0.2]{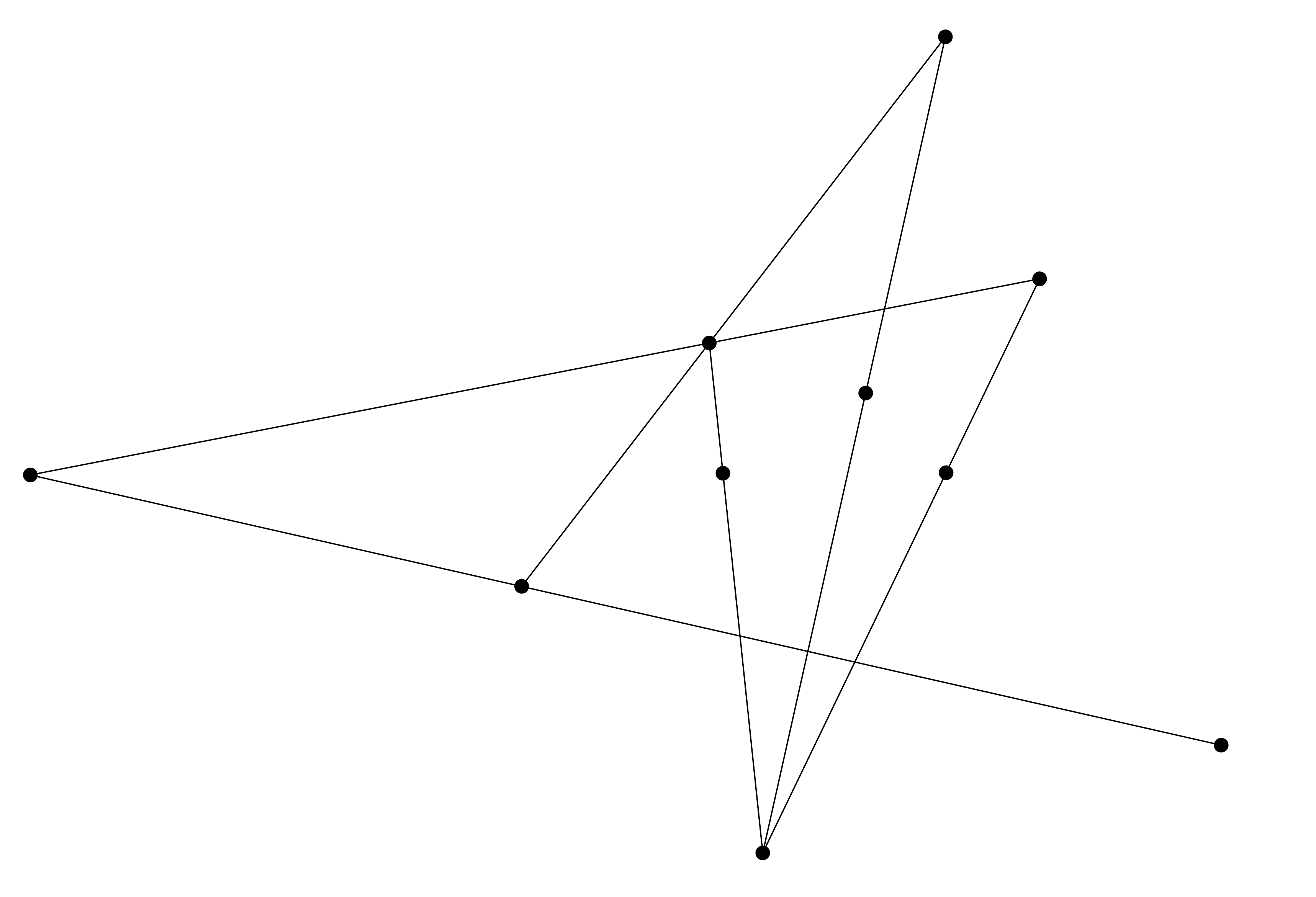}
      \end{tabular}
  \caption{Desargues configuration and a minimally rigid subconfiguration.}
  \label{Desargues}
  \end{center}
\end{figure}
 
\begin{example}[Desargues configuration]
The Desargues configuration, to the left in Figure \ref{Desargues}, is a $10_3$-configuration. As $2|P|-3=17$ is not divisible by $2k-3=3$, a minimally infinitesimally rigid subconfiguration of the Desargues configuration cannot satisfy the count in Corollary \ref{Point-line count}. 

Figure \ref{Desargues} also shows, to the right, a minimally infinitesimally rigid subconfiguration of the Desargues configuration. The subconfiguration satisfies $3|L|=2|P|-2$, but any strict subset $L' \subset L$ satisfies $3|L'| \leq 2|P'|-3$, where $P'$ is the set of points generated by $L'$.
\label{Desargues ex}
\end{example}

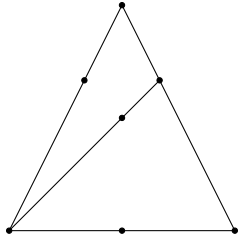
\begin{figure}
	\begin{center}
	\begin{tikzpicture}
	\filldraw[black] (0,0) circle (1pt);
	\filldraw[black] (-1.5,-3) circle (1pt);
	\filldraw[black] (1.5,-3) circle (1pt);
	\filldraw[black] (0.5,-1) circle (1pt);
	\filldraw[black] (-0.5,-1) circle (1pt);
	\filldraw[black] (0,-3) circle (1pt);
	\filldraw[black] (0,-1.5) circle (1pt);

	\draw[thin] (0,0) -- (1.5,-3) -- (-1.5,-3)--(0,0);
	\draw[thin] (-1.5,-3) -- (0.5,-1);
	\end{tikzpicture}
	
	\end{center}
	\caption{Minimally infinitesimally rigid example with $3|L|=2|P|-2$.}
	\label{2P-2 Example}
\end{figure}

\begin{example}
Consider the rod configuration in Figure \ref{2P-2 Example}. Clearly, it is minimally infinitesimally rigid, as removing a line would either leave a point not incident to any line, or leave a line with two points that are not incident to any other line, making the configuration flexible in either case. In this rod configuration, $P=7$ and $L=4$, so $3L=2P-2$.
\label{2P-2 ex}
\end{example}

For the incidence geometries in Example \ref{Desargues ex} and Example \ref{2P-2 ex}, we can consider $L'=L \setminus \{\ell\}$, where $\ell$ is any line with exactly one point incident only to $\ell$, and the points $P'$ generated by $L'$. In both cases we obtain an incidence geometry satisfying the condition of Corollary \ref{Point-line count}. This is the idea behind the next proposition.

\begin{prop}
Let $S=(P,L,I)$ be a $3$-uniform incidence geometry. 

\begin{enumerate}
	\item If $3|L|=2|P|-2$ and $3|L'|\leq 2|P'|-3$ for all $L'\subsetneq L$, with $P'$ the point set covered by $L'$, then $S$ has a realization as a minimally infinitesimally rigid rod configuration.
	\item If $3|L|=2|P|-1$, $3|L'| \leq 2|P'| -2$ for all $L' \subsetneq L$ and $3|L'| \leq 2|P'|-3$ for all $L' \subset L$ with $|L'| \leq |L|-2$ then $S$ has a realization as a minimally infinitesimally rigid rod configuration.
\end{enumerate}

\label{2|P|-2}
\end{prop}

\begin{proof}
\begin{enumerate}
	\item First we will show that there is a line $\ell \in L$ such that if $L'=L \setminus \{\ell\}$, then $|P'|=|P|-1$. 

Suppose no such line exists. Then for any line $\ell \in L$ and $L'=L \setminus \{ \ell \}$, $|P'|=|P|$ or $|P'|=|P|-2$. Take a line $\ell$ so that $|P'|=|P|-2$. Then $3|L'|=3|L|-3=2|P|-2-3=2|P|-5=2|P'|-1$, which contradicts our assumption. Hence $L$ cannot contain any such line, and all lines $\ell \in L$ must be so that $|P|=|P'|$. In that case, any point is incident to at least two lines, so $|I| \geq 2|P|$. As $|I|=3|L|$ for any 3-uniform incidence geometry, it follows that $3|L| \geq 2|P| > 2|P|-2$ which, again, contradicts our assumptions. Hence there must be a line $\ell$ so that $|P'|=|P|-1$. 

Let $L'=L \setminus \{ \ell \} $ and consider the subgeometry $S'=(P',L',I')$. It will satisfy $3|L'|=3|L|-3=2|P|-2-3=2|P'|-3$, and for any subset $L" \subseteq L'$, its generated incidence geometry $S"=(P",L",I")$ will satisfy the inequality $3|L"| \leq 2|P"|-3$, as $L"$ is a strict subset of $L$. It follows from Corollary \ref{Point-line count} that $S'$ has a realization as a minimally infinitesimally rigid rod configuration. 

It is possible to add the line $\ell$ between the appropriate points in $P'$ and to then add a point on $\ell$. This cannot make the configuration infinitesimally flexible, so the result is an infinitesimally rigid rod configuration realizing $S$.

By Proposition \ref{minrig} $S$ this infinitesimally rigid rod configuration is also minimally infinitesimally rigid.

	\item As in case 1, we can prove that $L$ contains a line $\ell$ such that if $L'=L\setminus \{\ell\}$, then $|P'|=|P|-1$.

Let $L'= L \setminus \{ \ell \}$ and consider $S'=(P',L',I')$. Then $3|L'|=3|L|-3=2|P|-1-3=2|P'|-2$. Consider a subset $L" \subsetneq L'$ and its generated incidence structure $S"=(P",L",I")$. Then $3|L"| \leq 2|L"|-3$, as $L"$ is a subset of $L$ with $|L"|\leq |L|-2$.

By case 1, $S'$ has a realization as a minimally infinitesimally rigid rod configuration. It is again always possible to add the line $\ell$ between the appropriate points in $P'$ and to add the remaining point on $\ell$. The result is an infinitesimally rigid rod configuration realizing $S$.

By Proposition \ref{minrig} $S$ this infinitesimally rigid rod configuration is also minimally infinitesimally rigid.
\end{enumerate}
\end{proof}

The incidence geometries that satisfy either set of conditions given in Proposition \ref{2|P|-2} can all be constructed from an incidence geometry that has a realization as a minimally infinitesimally rigid rod configuration satisfying the conditions of Corollary \ref{Point-line count} by adding a lines incident to two existing points, and one new point only incident to the added line. Those incidence geometries that satisfy the first set of conditions can be constructed by adding one line in this way to an incidence geometry satisfying the conditions in Corollary \ref{Point-line count}, and those that satisfy the second set of conditions can be constructed by adding two lines.

\begin{figure}
  \begin{center}
  \includegraphics[scale=0.2]{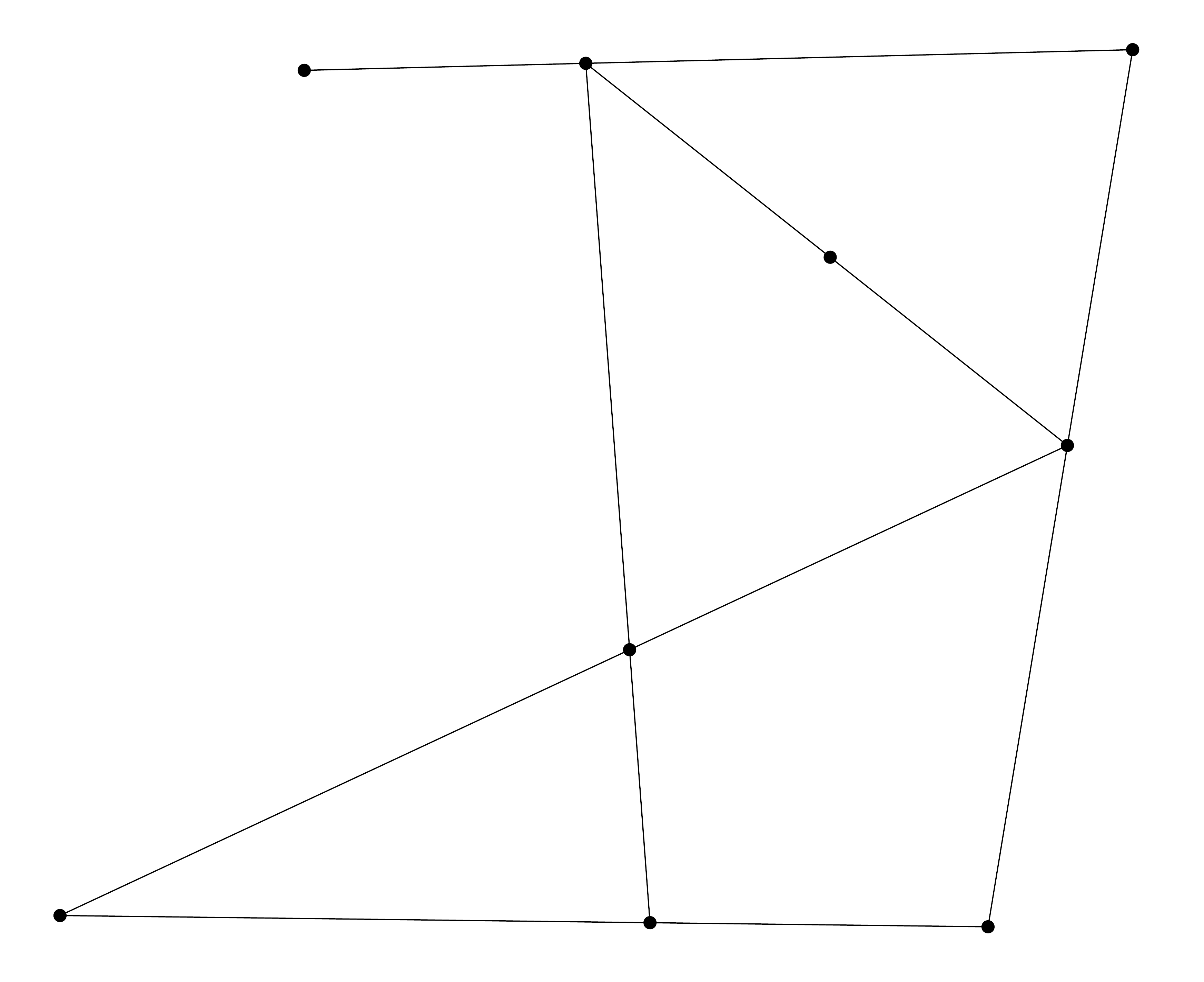}
  \caption{A rod configuration with $3L=2P$ that is rigid but not minimally rigid.}
  \label{2P NM}
  \end{center}
\end{figure}

\begin{example}
The rod configuration $(P,L,I)$ in Figure \ref{2P NM} has $|L|=6$ and $|P|=9$, so $3|L|=2|P|$. Any rod that does not have a point incident only to that rod can be removed without the rod configuration becoming infinitesimally flexible.  

Furthermore, removing either of the two lines that do have a point incident only to that line results in a rod configuration with $|L|=5$ and $|P|=8$ that satisfies the second set of conditions given in Proposition \ref{2|P|-2}. 

What this means is that the rod configuration in Figure \ref{2P NM} can be constructed from a rod configuration that satisfies the conditions of Corollary \ref{Point-line count} by adding lines between existing points and points incident only to those lines; similarly to how those incidence geometries that satisfy either set of conditions in Proposition \ref{2|P|-2} were constructed, only in this case, three lines are added. However, unlike the rod configurations that satisfy the conditions given in Proposition \ref{2|P|-2}, it is not minimally rigid. Therefore, the method of constructing minimally rigid rod configurations by adding lines with a single point only incident to that line is not guaranteed to work if three or more lines are added.
\end{example}

Corollary \ref{Point-line count} and Proposition \ref{2|P|-2} do not characterize the minimally infinitesimally rigid rod configurations; i.e. it is not true that a rod configuration is minimally infinitesimally rigid if and only if it satisfies the counts of either Corollary \ref{Point-line count} or Proposition \ref{2|P|-2}, as seen in the following example.

\begin{example}
Consider the rod configuration on the left in Figure \ref{2P min}. In this rod configuration, $3|L|=2|P|$. Yet it is easy to see that it is minimally infinitesimally rigid. 

An infinite sequence of minimally infinitesimally rigid rod configurations can be constructed by extending the leftmost rod configuration Figure \ref{2P min} by the structure in Figure \ref{min extension}. The rightmost rod configuration in Figure \ref{2P min} shows the next rod configuration in this sequence. Any rod configuration in this sequences has four more points and three more lines than the previous one. 

Suppose that a rod configuration in this sequence satisfies $3|L|=2|P|+k$. Then the next rod configuration in the sequence, realizing an incidence geometry $S'=(P',L',I')$, satisfies $3|L'|=3|L|+9=2|P|+k+9=2|P'|+k+1$, where the last equality holds since $|P'|=|P|+4$. All rod configurations in this sequence are minimally infinitesimally rigid, so there is a minimally infinitesimally rigid rod configuration that satisfies $3|L|=2|P|+k$ for any $k \geq 0$.
\end{example}

\begin{figure}
  \begin{center}
    \begin{tabular}{cc}
    \includegraphics[scale=0.3]{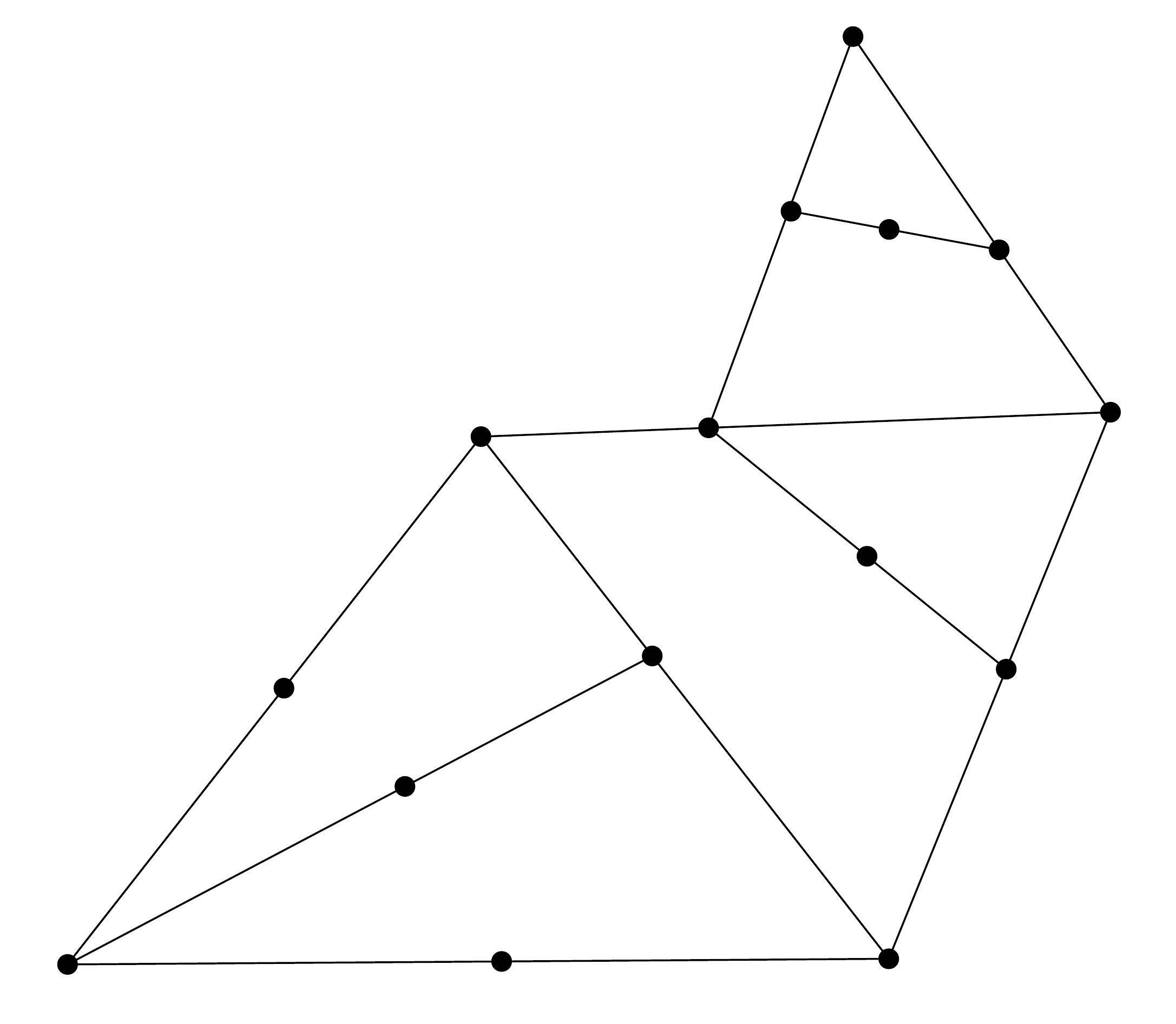}&  \includegraphics[scale=0.3]{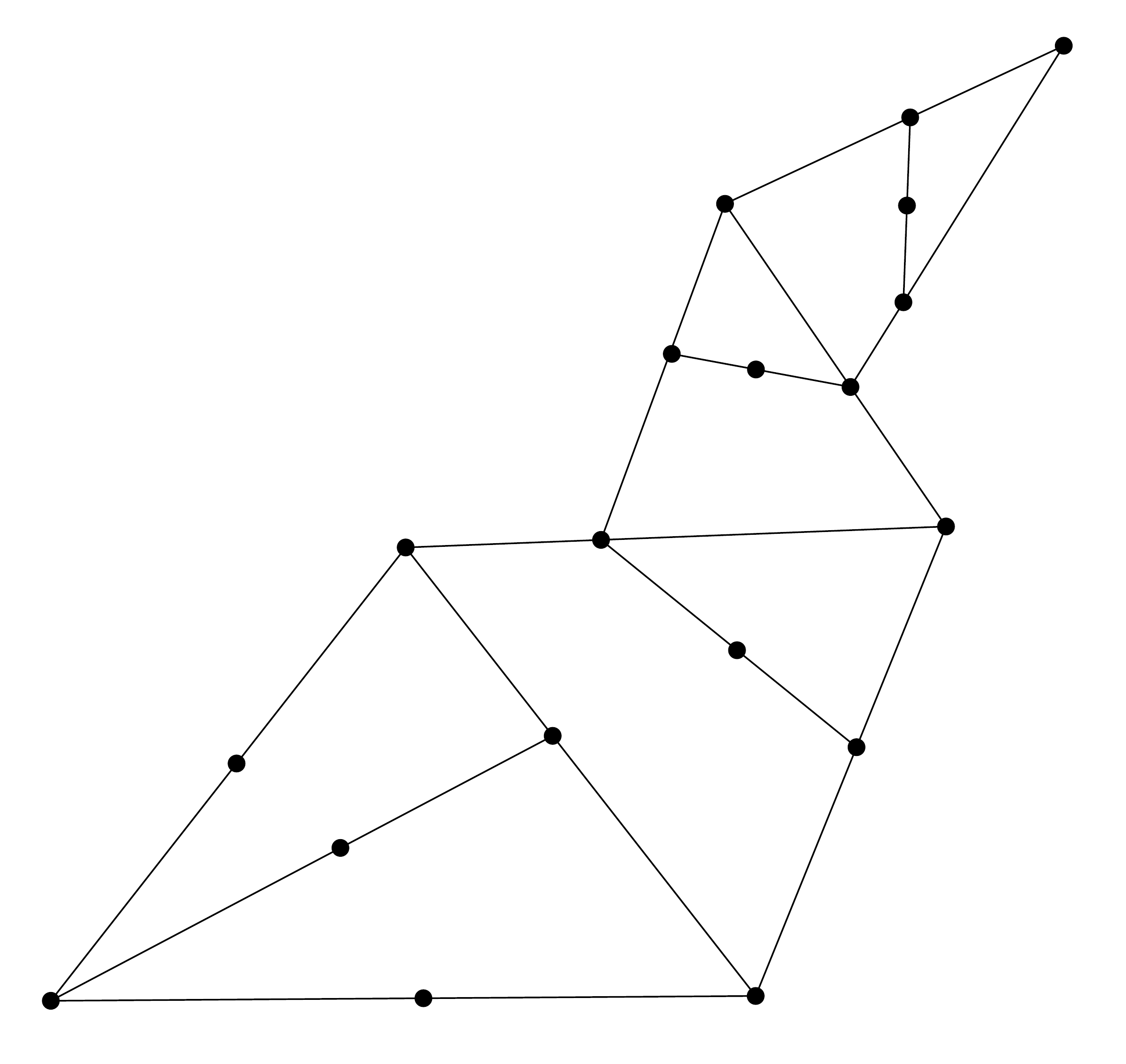}
    \end{tabular}
  \caption{Two minimally rigid rod configurations. The leftmost satisfies  $3|L|=2|P|$ and the rightmost satisfies $3|L|=2|P|+1$.}
  \label{2P min}
  \end{center}
\end{figure}

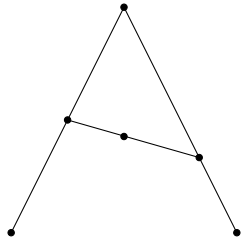
\begin{figure}
	\begin{center}
	\begin{tikzpicture}
	\filldraw[black] (0,0) circle (1.2pt);
	\filldraw[black] (-1.5,-3) circle (1.2pt);
	\filldraw[black] (1.5,-3) circle (1.2pt);
	\filldraw[black] (1,-2) circle (1.2pt);
	\filldraw[black] (-0.75, -1.5) circle (1.2pt);
	\filldraw[black] (0, -1.72) circle (1.2pt);

	\draw[thin] (0,0) -- (1.5,-3);
	\draw[thin] (-1.5,-3)--(0,0);
	\draw[thin] (1,-2) -- (-0.75, -1.5);
	
	\end{tikzpicture}
	
	\end{center}
	\caption{The structure used in the construction of an infinite family of minimally rigid rod configurations}
	\label{min extension}
\end{figure}

Finally, Figure \ref{9L15P} shows a minimally infinitesimally rigid rod configuration with 15 points and 9 lines. Recall that the leftmost rod configuration in Figure \ref{2P min} is minimally infinitesimally rigid with 15 points and 10 lines. This illustrates a difference to graphs; any minimally rigid graph with a specified number of vertices will have the same number of edges, but it is not true that all minimally rigid rod configurations with some fixed number of points will have the same number of lines.

\begin{figure}
	\begin{center}
	\includegraphics[scale=0.3]{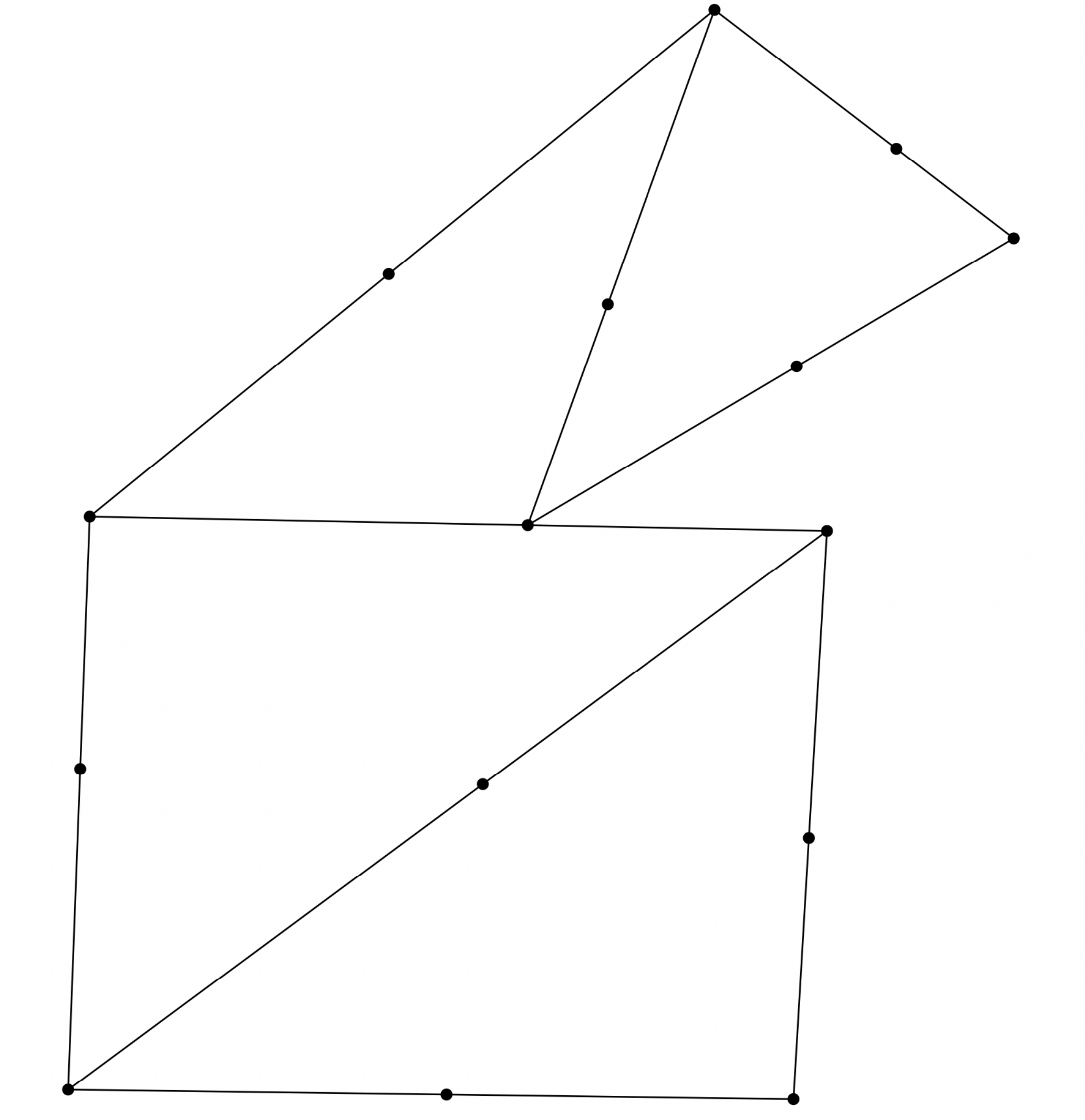}
	\end{center}
	\caption{A rod configuration with 9 lines and 15 points.}
	\label{9L15P}
\end{figure}

\section{Flexible $v_3$-configurations of rods and points in the plane}
\label{Flexible}

\subsection{Infinite families of flexible $v_3$-configurations with the motions of the polygons}

By a theorem of Steinitz, every combinatorial $v_3$-configuration can be realized as a rod configuration if one incidence is removed \cite{steinitz,grunbaum}.
The two smallest $v_3$-configurations have $7$ and $8$ points, and they are known as the Fano plane and the M\"obius-Kantor configuration, respectively. Neither of them can be realized geometrically as rod configurations. Figure~\ref{M-K} shows a geometric realization of the M\"obius-Kantor configuration with one line removed.

\begin{figure}
\begin{center}
	\includegraphics[scale=0.3]{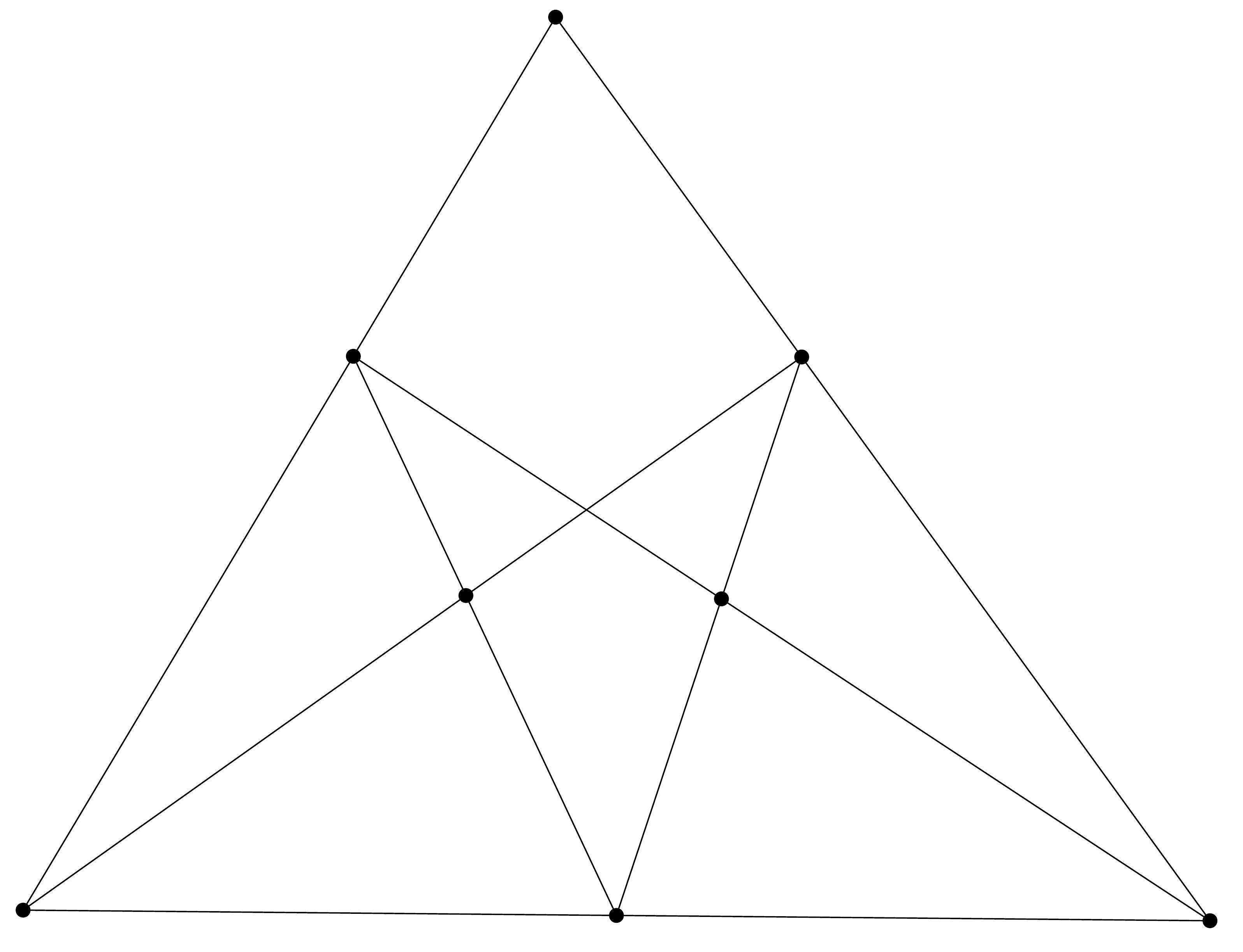}
\end{center}
\caption{The M\"obius-Kantor configuration with a line removed}
\label{M-K}
\end{figure}

For $v=9$ there are $3$ distinct combinatorial $v_3$-configurations, all geometrically realizable. For $v=10$ there are $10$ distinct combinatorial $v_3$-configurations, one of which is not geometrically realizable as a rod configuration. For $v=11$ and $v=12$ all combinatorial $v_3$-configurations have geometric realizations. 
It is however generally believed that, asymptotically in $v$, the combinatorial configurations that cannot be realized as rod configurations form a large portion of the set of all combinatorial configurations.

There are, however, not only realizable combinatorial $v_3$-configurations for all $v$ large enough, but even flexible geometric $v_3$-configurations, as we prove in the following theorem.

\begin{theorem}
  \label{thm:eventually}
  There are flexible geometric $v_3$-configurations for all $v \geq 28$.
\end{theorem}
\begin{proof}
  Take $n\geq 2$ disjoint combinatorial $v_3$-configurations $C_0,\dots,C_{n-1}$ (possibly $C_i\sim C_j$ with $i\neq j$) that cannot be realized geometrically as rod configurations.

  Use Steinitz's Theorem \cite{steinitz} to find disjoint geometric realizations of the $n$ disjoint copies of combinatorial configurations with one incidence removed in each configuration. Each configuration $C_i$ then contains a line $\ell_i$ with two incidences and a point $p_i$ with two incidences. Move each configuration independently using the rigid motions of the plane so that the geometric realization of the line $\ell_i$ passes through the geometric realization of the point $p_{i+1\pmod{n}}$.

  Combinatorially, this is a variant of  repeated use of the ``incidence switch'' \cite{stokes,grunbaum} on the pairs $(C_{i},C_{i+1\pmod{n}})$, which constructs a connected combinatorial configuration from the $n$ disjoint copies.

  Applying this construction to $a$ copies of the Fano plane and $b$ copies of the M\"obius-Kantor configuration gives us a geometric $v_3$-configuration for all parameters $v$ of the form $v=7a+8b$, $a,b\in\mathbb{Z}_{\geq 0}$. This is a numerical semigroup generated by the two coprime natural numbers $7$ and $8$. The largest natural number not on the form $7a+8b$ is $41$ (a.k.a. the Frobenius number of the numerical semigroup $\left\langle 7, 8\right\rangle)$. This bound can be improved by using configurations with other parameters. For example, the Frobenius number of $\langle 7,8,10\rangle$ is $19$, giving the bound $v\geq 20$.

  The rod configurations constructed in this way from $n$ rod configurations, isomorphic to either the Fano plane or the M\"obius-Kantor configuration,  have the motions of the $n$-gon. Therefore they are flexible if $n\geq 4$ (with $n-3$ degrees of freedom) and rigid if $n\leq 3$. It can easily be checked that all integers larger or equal to $42$ are of the form $7a+8b$ with $a+b\geq 4$.
  If we instead consider $\langle 7,8,10\rangle$ the bound becomes $28$, but it is likely that this bound can be improved. 
\end{proof}

For illustration of the construction used in the proof of Theorem \ref{thm:eventually}, see Figure \ref{7n_3}. It is possible that the bound $v \geq 28$ could be improved further, for example by in the construction also using geometrically realizable $v_3$-configurations with an incidence removed.

\begin{figure}
  \begin{center}
    \begin{tabular}{ccc}
	\includegraphics[scale=0.15]{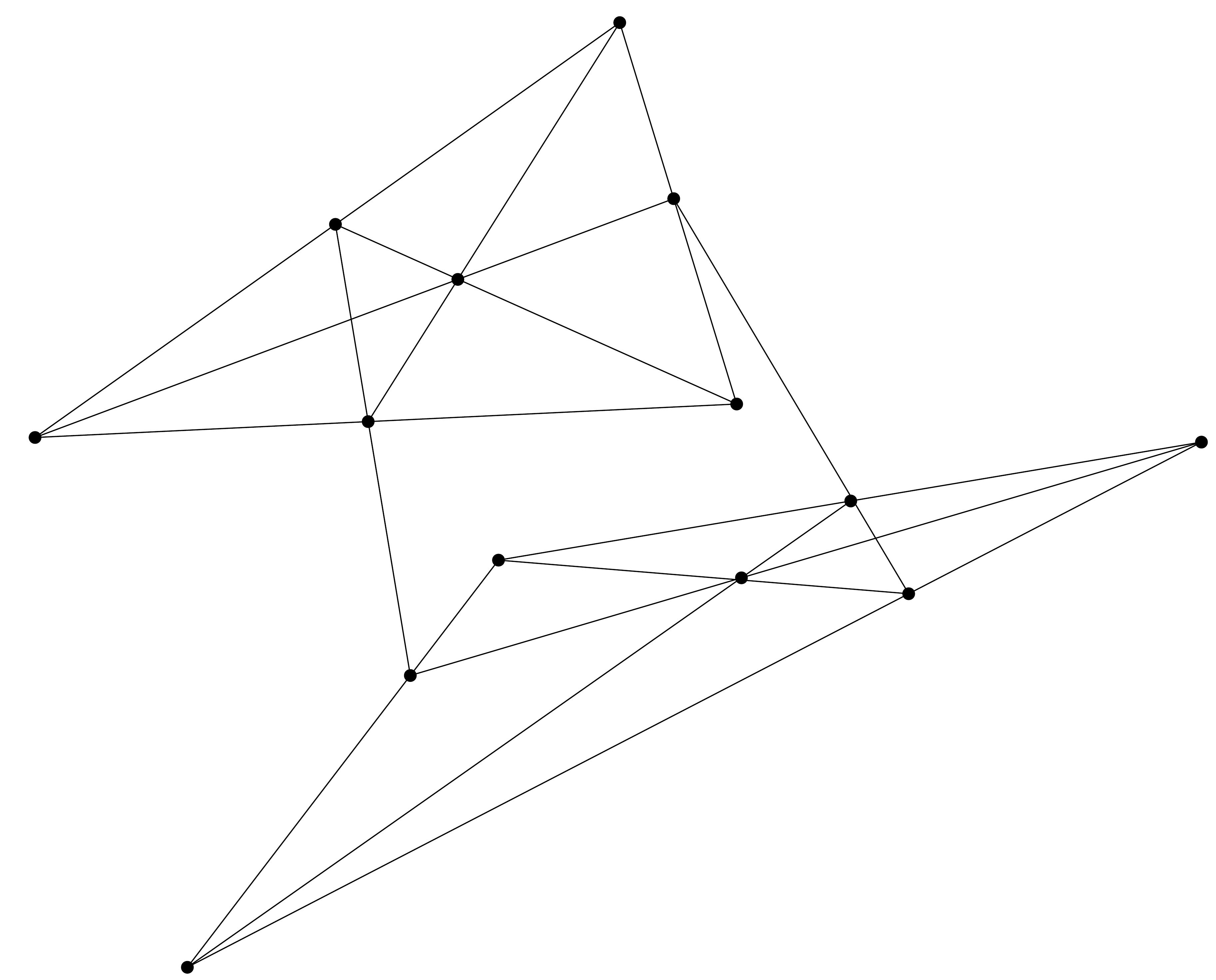}&
	\includegraphics[scale=0.18]{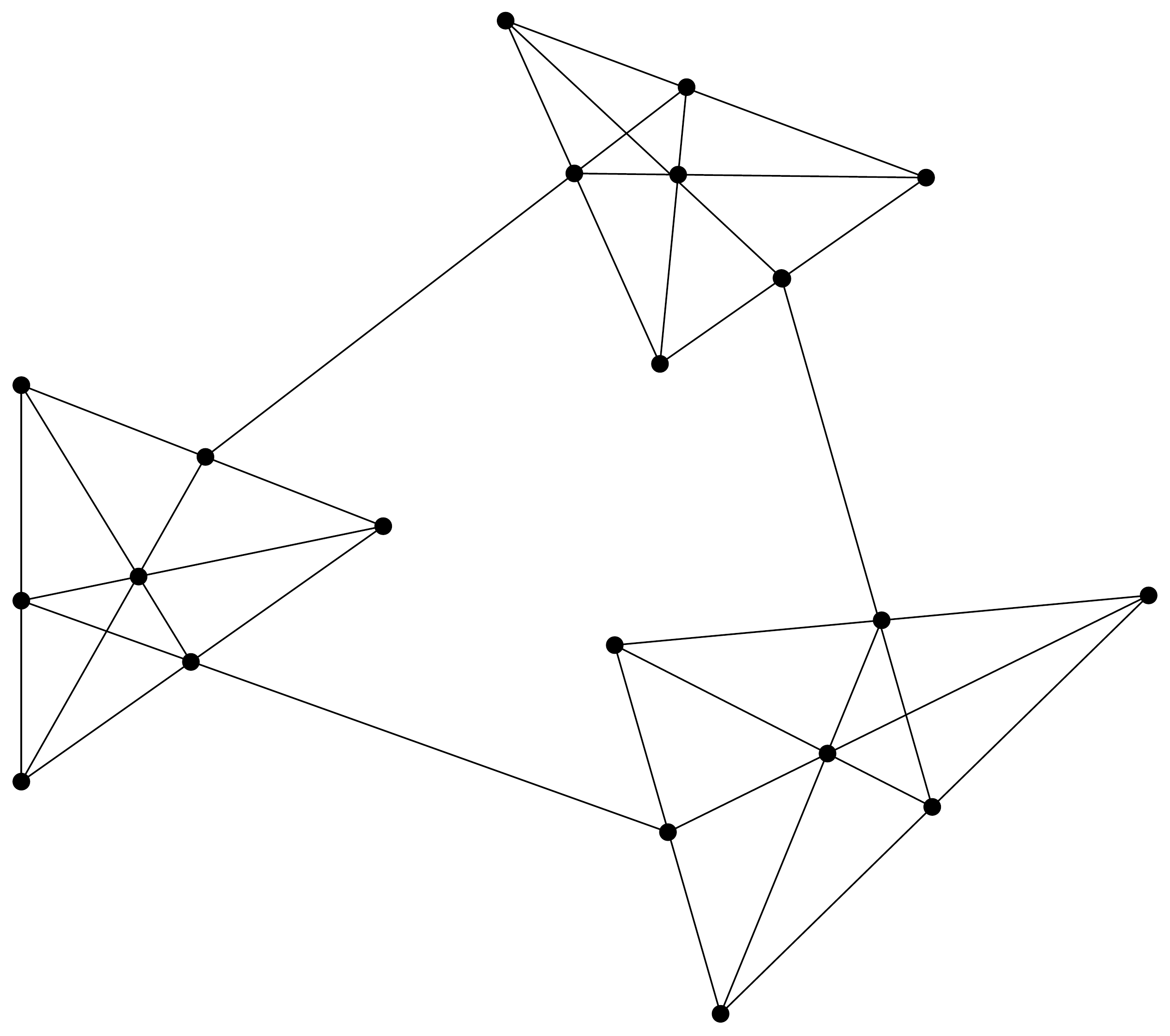}&
	\includegraphics[scale=0.18]{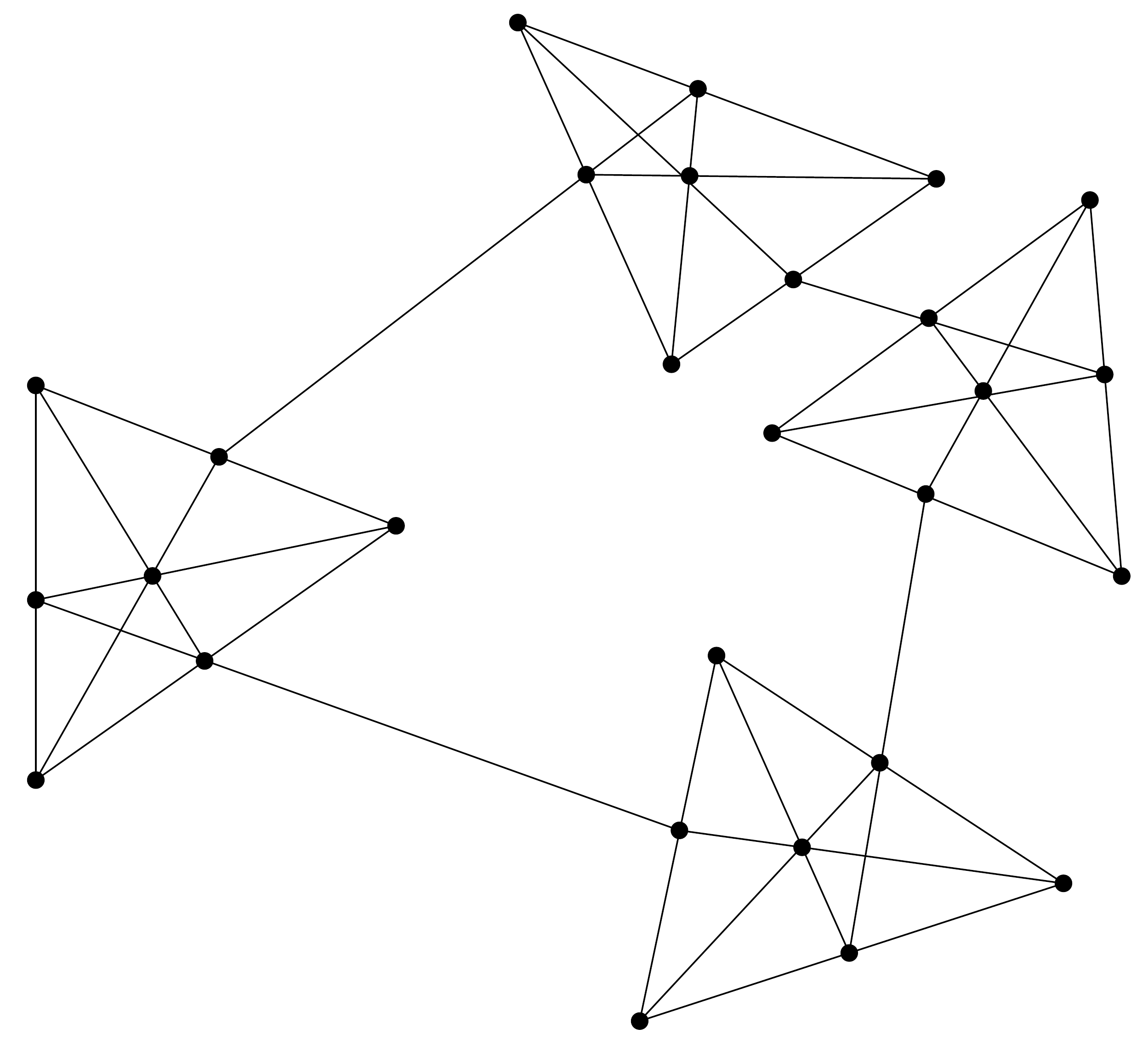}
    \end{tabular}
\end{center}
	\caption{The first three configurations in the infinite sequence of $7n_3$-configurations with the motions of the $n$-gon}
	\label{7n_3}
\end{figure}

The infinite family of flexible rod configurations constructed in the proof of Theorem~\ref{thm:eventually} can be modified slightly to give another infinite family of rod $v_3$-configurations. This family is illustrated in Figure \ref{fig:square_full}, which shows a $32_3$ point-line incidence geometry constructed from a square by adding a copy of the Fano plane with one line removed to each of its edges.

\begin{figure}
\begin{center}
	\includegraphics[scale=0.3]{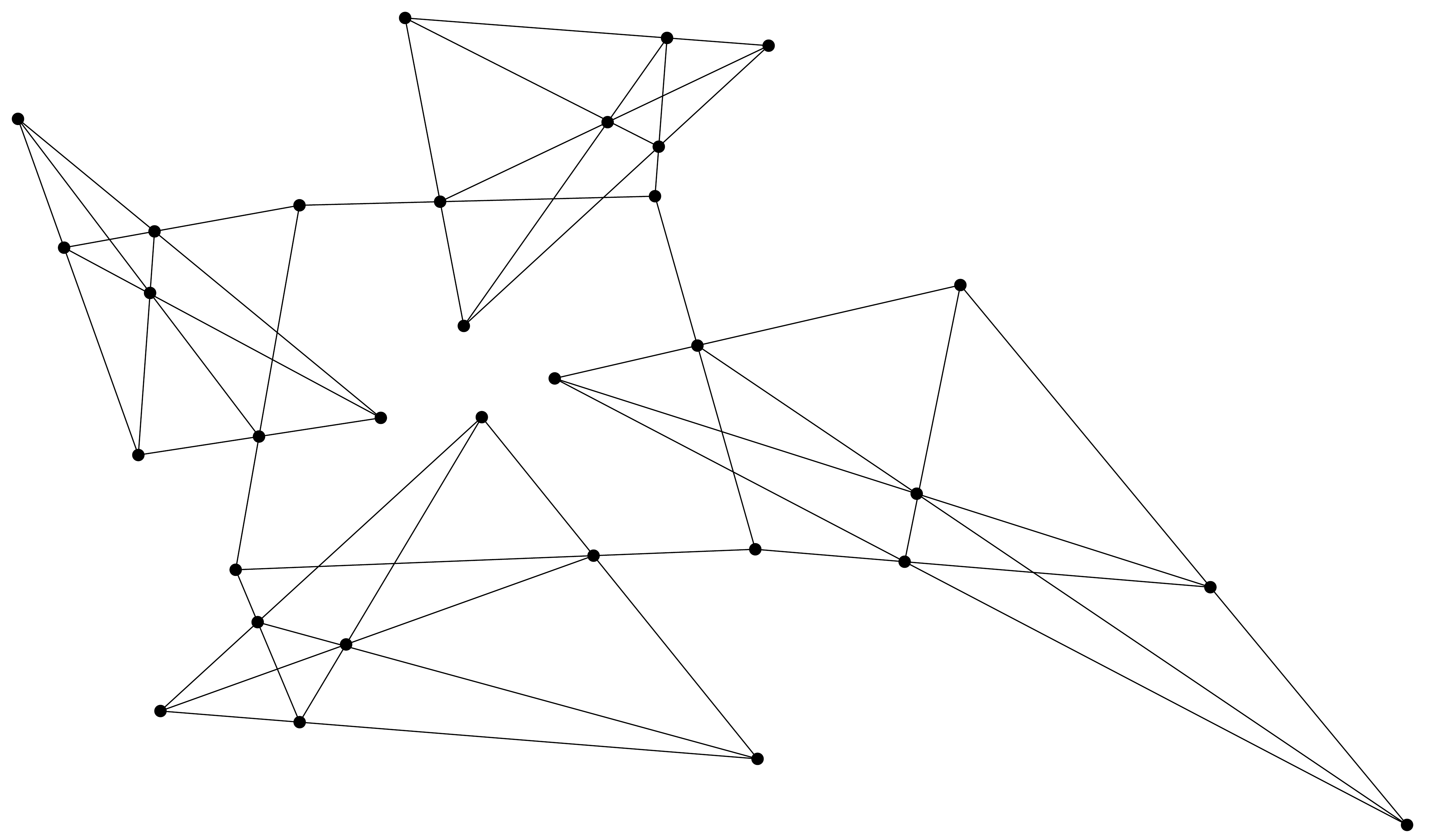}
\end{center}
	\caption{Another flexible configuration with the motions of a square}
	\label{fig:square_full}
\end{figure}

\subsection{Configurations that are flexible in special positions}

The point coordinates of the configurations in Figures \ref{7n_3} and \ref{fig:square_full} are not algebraically independent, so the configurations are not in generic position (according to the usual definition for graphs). The vertices of the square are in generic position, however, and more importantly, the motion will not disappear if the vertices of the square are placed elsewhere.

The $45_3$-configuration in Figure~\ref{fig:4}, on the contrary, features a motion that only exists because certain lines are parallel. It is constructed from six Fano planes, each with one line removed (the line that is realized as a circle in the common representation of the Fano plane in the real plane). These are joined together with two parallel grid structures that move independently. If the lines that are parallel in each of the two grid structures were not parallel, the configuration would not move. 

\begin{figure}
\begin{center}
  \includegraphics[scale=0.3]{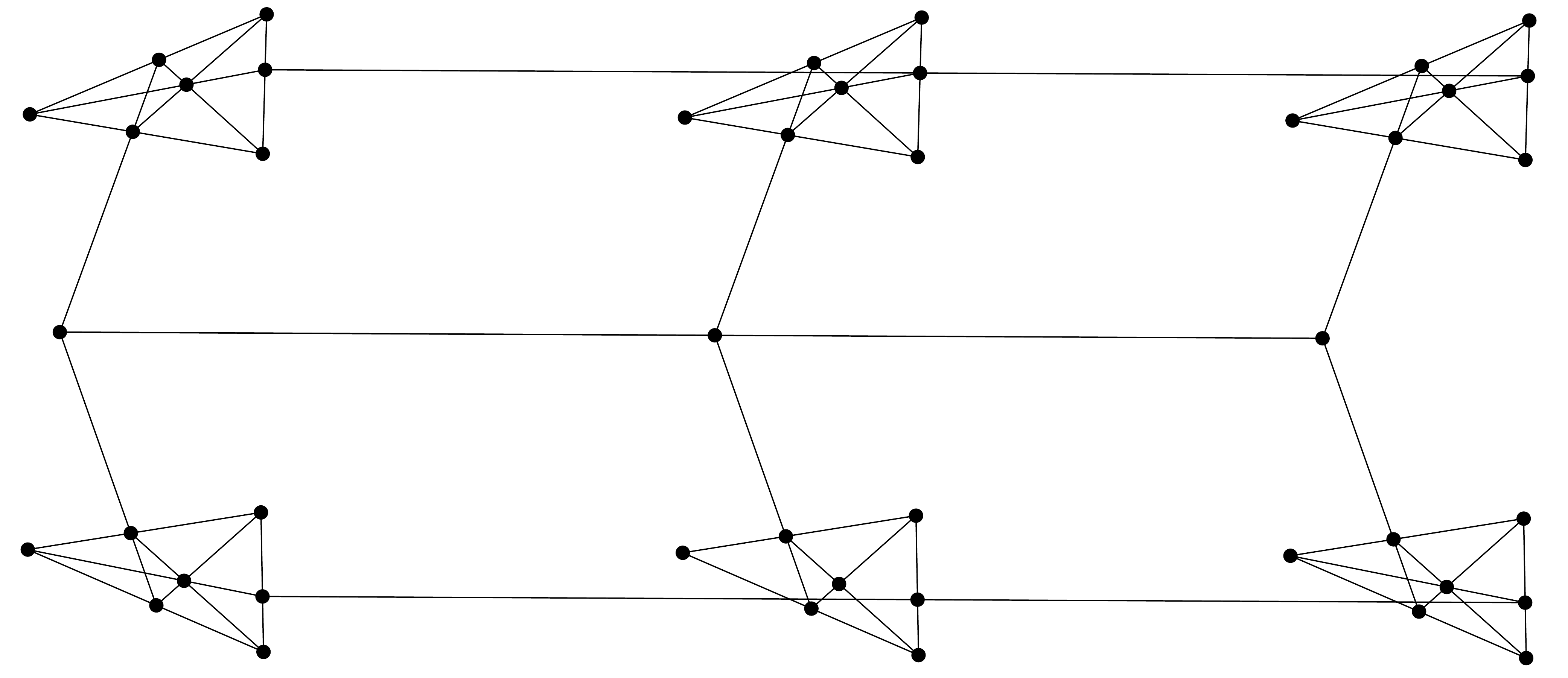}
  \end{center}
\caption{A flexible $45_3$-configuration with a motion in special position }
\label{fig:4}
\end{figure}

As another example, we consider the Gray configuration, which is a $27_3$-configuration that has a geometric realization with the points on a square lattice in dimension three. A planar point-line configuration is obtained by projecting it on a plane, see Figure~\ref{fig:5}. The configuration in dimension three is flexible, as is its planar projection. Again, the flexibility is due to the parallel positions of the lines; if some of the lines were not in parallel position, the configuration would be rigid. It is easy to see that the generalized Gray configuration, obtained through the same construction but starting with a square lattice in dimension $n$, has the same property. 

\begin{figure}
\begin{center}
\includegraphics[scale=0.3]{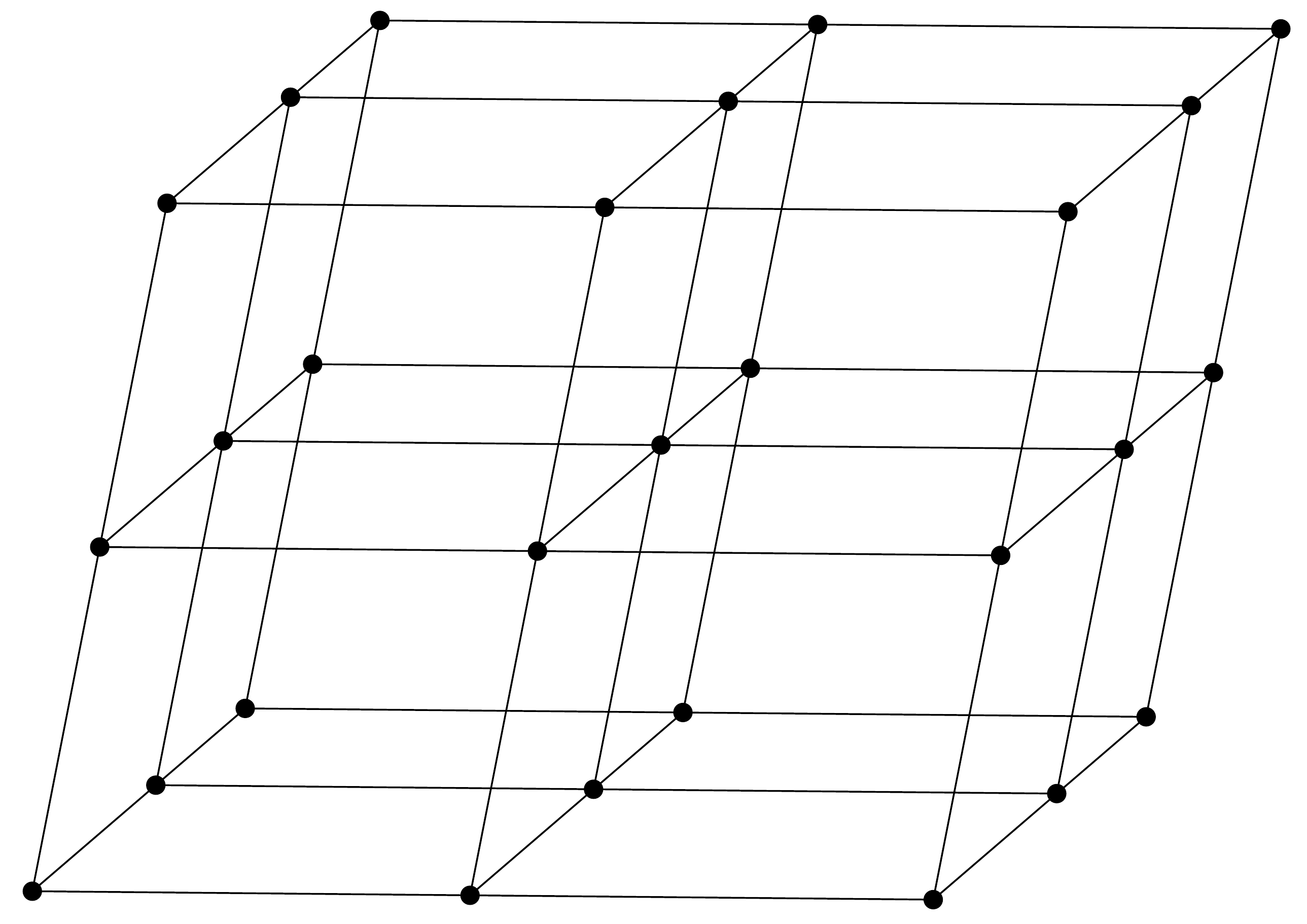}
\end{center}
	\caption{The Gray configuration}
	\label{fig:5}
\end{figure}

\subsection{Global flexibility of rod configurations}

Frameworks of graphs that admit a motion that preserves the lengths of edges, other than the Euclidean motions, are called globally flexible. By this definition, continuously flexible graphs are also globally flexible. There are also frameworks of graphs that are continuously and infinitesimally rigid, but still admit a non-continuous motion. A framework that is not globally flexible is globally rigid. The globally rigid graphs have been characterized by Connelly \cite{connelly} and Jackson and Jord\'an \cite{JacJorGlobal}. 

There is a natural extension of this notion to rod configuration. A rod configuration is said to be globally flexible if it admits a (possibly non-continuous) motion that preserves the pairwise distance between collinear points, other than the Euclidean motions. As with graphs, continuously flexible rod configurations are \emph{a fortiori} globally flexible. A rod configuration that is not globally flexible is said to be globally rigid.

Some of the continuously rigid rod configurations we have presented, such as the two joined Fano planes, and the triangle with a reduced Fano plane on each side (left and center in Figure \ref{7n_3}), are globally flexible. The triangle with a reduced Fano plane on each side admits a reflection of one of the reduced Fano planes in the line it is attached to. In contrast, for example, the Pappus configuration (Figure \ref{Pappus subconfiguration}) is not globally flexible.

Determining conditions on incidence geometries that have realizations as globally rigid rod configurations in the plane is a possible line for further research.

\section{Conclusions and open problems}

In this article, we have surveyed realizations of incidence geometries as rod configurations in the plane and their rigidity properties. 

Theorem \ref{Whiteley5.4} gives a combinatorial condition for an incidence geometry to have a realization as a minimally infinitesimally rigid rod configuration in the plane.  By adding certain rods to configurations that satisfy that condition, we derived further combinatorial conditions that imply that an incidence geometry has minimally infinitesimally rigid rod configurations in the plane. We have also provided examples that show that these conditions are not sufficient; in fact, minimally infinitesimally rigid rod configurations can be arbitrarily far from satisfying the  combinatorial conditions. 

Further, we provide examples of flexible geometric $v_3$-configurations. There are infinitely many such $v_3$-configurations, and if $v$ is sufficiently large, there is at least one flexible $v_3$-configuration. 

We conclude with some open problems.

\begin{enumerate}
  \item We know that any infinitesimally rigid graph $G=(V,E)$ has (at least one) spanning minimally infinitesimally rigid subgraph $G'=(V,E')$. All minimally infinitesimally rigid graphs on a given number of vertices has the same number of edges.

  Similarly, any infinitesimally rigid rod configuration has at least one spanning minimally infinitesimally rigid subconfiguration. As we have seen, not all minimally infinitesimally rigid subconfigurations on a given number of points have the same number of lines. It would be interesting to see if there are rigid rod configurations with more than one spanning minimally infinitesimally rigid subconfigurations, such that at least two of them have different number of lines.

  \item In Section \ref{Minimal rigidity} we saw that the minimally rigid rod configurations that satisfy either set of conditions in Proposition \ref{2|P|-2} can be constructed from a rod configuration satisfying Corollary \ref{Point-line count} by adding a line between two existing points and a point incident only to those lines. We have also seen that if we add more than two lines in this way, then we can no longer guarantee that the result is minimally infinitesimally rigid. However, it is possible to add more than two lines in this way to a given minimally infinitesimally rigid rod configuration and obtain a rod configuration that is minimally rigid.
It would be interesting to find an upper bound for the number of lines one can add in this way to a rod configuration satisfying Corollary \ref{Point-line count} before one can guarantee that the resulting rod configuration is not minimally infinitesimally rigid.

\item In Section \ref{Minimal rigidity}, we saw that there are infinitesimally rigid geometric $v_3$-configurations, and in Section \ref{Flexible} we saw that there are infinitesimally flexible $v_3$-configurations. However, we still do not have any examples of minimally rigid geometric $v_3$-configurations.
\end{enumerate}

\section{Acknowledgements}

The work has been supported by the Knut and Alice Wallenberg Foundation Grant 2020.0001 and 2020.0007.

We want to thank Brigitte Servatius and Walter Whiteley for useful discussions.


\begin{thebibliography}{99}

\bibitem{stokes} M.\ Bras-Amor\'os and K.\ Stokes. The semigroup of combinatorial configurations. Semigroup Forum, 84 (2012) 91–96.
\bibitem{Roth} L. \ Asimov and B. \ Roth. The rigidity of graphs. Trans. Amer. Math. Soc, vol. 245 (1978) 279--289.
\bibitem{connelly} R.\ Connelly. Generic global rigidity. Discrete Comput. Geom. 33:4 (2005) 549--563.
\bibitem{Crapo84} H.\ Crapo. Concurrence geometries. Advances in Mathematics 3 (1984) 75--95.
\bibitem{Crapo85} H.\ Crapo. The combinatorial theory of structures. In: Matroid theory ({S}zeged, 1982). Colloq. Math. Soc. J\'{a}nos Bolyai, North-Holland, Amsterdam (1985) 107--213.
\bibitem{CraWhi} H.\ Crapo and W.\ Whiteley. Plane self stresses and projected polyhedra. I. The basic pattern. Structural Topology, 20 (1993) 55--78.
\bibitem{Frameworks2019} Y.\ Eftekhari, B.\ Jackson, A.\ Nixon, B.\ Schulze, S.\ Tanigawa, W.\ Whiteley. Point-hyperplane frameworks, slider joints, and rigidity preserving transformations.  J. Combin. Theory Ser. B, 135 (2019), 44–74.
\bibitem{Gluck} H.\ Gluck. Almost all simply connected closed surfaces are rigid. In: Geometric topology ({P}roc. {C}onf., {P}ark {C}ity, {U}tah, 1974). Springer, Berlin (1975) 225--239.
\bibitem{GraSerSer} J.\ Graver, B.\ Servatius and H.\ Servatius. Combinatorial rigidity. Graduate Studies in Mathematics, vol. 2. American Mathematical Society, Providence, RI,1993.
\bibitem{grunbaum} B.\ Gr\"unbaum. Configurations of Points and Lines. Graduate Studies in Mathematics, V. 103. AMS, Providence, Rhode Island (2009).
\bibitem{Izquierdo-Stokes} M.\ Izquierdo and K.\ Stokes. Isometric Point-Circle Configurations on Surfaces from Uniform Maps. In: Širáň J., Jajcay R. (eds) Symmetries in Graphs, Maps, and Polytopes (2016), SIGMAP 2014. Springer Proceedings in Mathematics \& Statistics, vol 159. Springer. 
\bibitem{JacJorGlobal} B.\ Jackson and T.\ Jord\'an. Connected rigidity matroids and unique realizations of graphs. J. Combin. Theory Ser. B. 94:1 (2005) 1--29. 
\bibitem{JacJor05} B.\ Jackson and T.\ Jord\'an. Rigid two-dimensional frameworks with three collinear points.  Graphs Combin. 21 (2005), no. 4, 427–444.
 \bibitem{JacJor08} B.\ Jackson and T.\ Jord\'an. Pin-collinear body-and-pin frameworks and the molecular conjecture. Discrete Comput. Geom. 40:2 (2008) 258--278
\bibitem{JacJor10} B.\ Jackson and T.\ Jord\'an. The generic rank of body-bar-and-hinge frameworks. European J. Combin. 31:2 (2010) 574--588.
\bibitem{KatTan11} N.\ Katoh and S.\ Tanigawa. A proof of the molecular conjecture. Discrete Comput. Geom. 45:4 (2011) 647--700
\bibitem{Laman} G.\ Laman. On graphs and rigidity of plane skeletal structures. Journal of Engineering Mathematics. \textbf{4} (1970) 331--340.
\bibitem{Graphs and Geometry} L.\ Lovász. Graphs and Geometry. American Mathematical Society Colloquium Publications, V.65. AMS, Providence, Rhode Island, 2019.
\bibitem{Projective Lens} A.\ Nixon, B.\ Schulze and W.\ Whiteley. Rigidity Through a Projective Lens (2021). Available online at "\url{https://www.researchgate.net/publication/350838924_Rigidity_Through_a_Projective_Lens}".
\bibitem{pisanski} T.\ Pisanski and B.\ Servatius. Configurations from a Graphical Viewpoint. Springer, 2013. 

\bibitem{PolGei1927} H.\ Pollaczek-Geiringer. \'Uber die Gliederung ebener Fachwerke. Zeitschrift f\"ur  Angewandte Mathematik und Mechanik (ZAMM). (1927)  58--72.
\bibitem{SerWhi} B.\ Servatius and W.\ Whiteley. Constraining plane configurations in computer-aided design: combinatorics of directions and lengths. Structural Topology, SIAM J. Discrete Math. 12:1 (1999) 136--153.
 \bibitem{steinitz} E.\ Steinitz. \"Uber die Construction der Configurationen $n_3$. Ph.\ D.\ Thesis, Breslau (1894).
\bibitem{Tay1984} T.S.\ Tay. Rigidity of multi-graphs. I. Linking rigid bodies in n-space. Journal of Combinatorial Theory, Series B. 36:1 (1984) 95-112.
\bibitem{Tay1989} T.S.\ Tay. Linking {$(n-2)$}-dimensional panels in {$n$}-space. {II}. {$(n-2,2)$}-frameworks and body and hinge structures. Graphs Combin. 5:3 (1989) 245--273.
\bibitem{TayWhi1984} T.S.\ Tay and W.Whiteley. Recent advances in the generic rigidity of structures. Structural Topology. 9 (1984) 31--38.
\bibitem{Whiteley88} W.\ Whiteley. The union of matroids and the rigidity of frameworks. SIAM J. Discrete Math. 1:2 (1988) 237--255.
\bibitem{Whiteley89} W.\ Whiteley. A matroid on hypergraphs, with applications in scene analysis and geometry.  Discrete \& Computational Geometry. An International Journal of Mathematics and Computer Science. 4 (1989) 278--301.
\bibitem{Whiteley1996} W.\ Whiteley. Some matroids from discrete applied geometry. In: Matroid theory ({S}eattle, {WA}, 1995). Amer. Math. Soc., Providence, RI (1996) 171--311.
\end{thebibliography}
\end{document}